\documentclass[reqno,12pt]{amsart}
\usepackage{amsmath,amsfonts,amsthm,amscd,amssymb,graphicx,enumerate}

\newtheorem{theorem}{Theorem}[section]

\newtheorem{lemma}[theorem]{Lemma}

\newtheorem{proposition}[theorem]{Proposition}

\newtheorem{corollary}[theorem]{Corollary}

\newtheorem{example}[theorem]{Example}

\newtheorem{remark}[theorem]{Remark}

\newenvironment{acknowledgements}[1][Acknowledgements:]{\begin{trivlist}
\item[\hskip \labelsep {\bfseries #1}]}{\end{trivlist}}

\begin{document}
\title{B-Valued Free Convolution for Unbounded Operators}
\author[Williams, John D.]{John D.\ Williams}\let\thefootnote\relax\footnotetext{Supported in part the Alexander von Humboldt Stiftung and the Simons Foundation.}
\address{J.\ Williams, Fachrichtung Mathematik, Universit\"at des Saarlandes,
Saarbr\"ucken, Germany. 66041.}
\email{williams@math.uni-sb.de} 
\let\thefootnote\relax
\maketitle

\begin{abstract}
Consider the $\mathcal{B}$-valued probability space $(\mathcal{A}, E, \mathcal{B})$, where $\mathcal{A}$ is a tracial von Neumann algebra.
We extend the theory of operator valued free probability to the algebra of affiliated operators $\tilde{\mathcal{A}}$. For a random variable  $X \in \tilde{\mathcal{A}}^{sa}$ we study the Cauchy transform $G_{X}$ and show that the von Neumann algebra $(\mathcal{B} \cup \{X\})''$ can be recovered from this function.  In the case where $\mathcal{B}$ is finite dimensional, we show that, when $X, Y \in \tilde{\mathcal{A}}^{sa}$ are assumed to be $\mathcal{B}$-free, the $\mathcal{R}$-transforms are defined on universal subsets of the resolvent and satisfy $$ \mathcal{R}_{X} + \mathcal{R}_{Y} = \mathcal{R}_{X + Y}. $$
Examples indicating a failure of the theory for infinite dimensional $\mathcal{B}$ are provided.
Lastly, we show that the  functions that arise as the Cauchy transform of affiliated operators are the limit points of the  Cauchy transforms of bounded operators in a suitable topology.
\end{abstract}

\section{Introduction}
The theory of free probability was initiated in \cite{Vadd} with Voiculescu's observation that random variables in free product C$^{\ast}$ algebras display a form of independence that allows them to be studied through probabilistic methods.  In developing a complete probability theory, it was necessary to extend these results to measures with unbounded support.  This was accomplished in Bercovici and Voiculescu's work \cite{BV}, which inspired this present paper and remains one of the strongest and clearest articles in free probability theory.

Voiculescu extended free probability to amalgamated free products of C$^{\ast}$-algebras in \cite{V2}, replacing states acting on these algebras with conditional expectations onto a distinguished subalgebra.  This theory has achieved remarkable growth in recent years and we refer to \cite{Speichop}
for an overview of the combinatorial approach to this subject and \cite{BPV}, \cite{PV},  \cite{BMS} and \cite{AWMon}  for some of the recent advances in this field.

The purpose of this paper is to define free convolution for the operator-valued equivalent of unbounded measures.  
In particular, we will show that these operations may be defined provided that we study operators affiliated with a tracial von Neumann algebra $\mathcal{A}$.  In studying these objects, we are left with very few tools since there is not  an appropriate notion of $\mathcal{B}$-valued measure theory that the author is aware of and these operators need not have any moments.  Thus, the only reasonable probabilistic object associated to these operators is the Cauchy transform, and the analysis of this function will form the basis of our study.  In this sense, this paper may be seen as a continuation of the author's work in \cite{Wil6}.

The first of the main results in this paper is Theorem \eqref{unboundiso} which asserts that the Cauchy transform encodes all operator algebraic information associated to the random variable.  This is an extension of the following theorem due to Bercovici and Voiculescu to the $\mathcal{B}$-valued case:
\begin{theorem}\cite[Theorem 4.6]{BV}
Let $(\mathcal{A},\tau)$ be a W$^{\ast}$-probability space, let $T_{i} \in \tilde{\mathcal{A}}^{sa}$ be free random
variables for $i = 1,2, \ldots , k$, and let $Q$ be a self adjoint polynomial in $k$ non-commuting variables.  Then the distributions of the 
random variable $Q(T_{1}, T_{2} , \ldots , T_{k})$ depends only on the distributions of $T_{1} , T_{2} , \ldots, T_{k}$.
\end{theorem}
We note that our theorem requires agreement of the Cauchy transform on fairly robust subsets of the resolvent.  
It was anticipated that the non-commutative upper half plane would be a sufficient domain to recover the random variable, but this is not the case as the counterexample  in Section \eqref{CounterE} demonstrates.

The next main result is Theorem  \eqref{MainTheorem} which states that the Cauchy transforms of unbounded operators may be characterized as the limiting functions of the Cauchy transforms of bounded operators.  

The last of the main results are Theorems \eqref{Rexist} and \eqref{Rsum} which form operator valued version of the following:
\begin{theorem}\cite[Corrolary 5.8]{BV}
Let $\mu_{1}$ and $\mu_{2}$ be probability measures on $\mathbb{R}$, and let $\mu = \mu_{1} \boxplus \mu_{2}$.  For each $\alpha > 0$ we have that $\varphi_{\mu} = \varphi_{\mu_{1}} + \varphi_{\mu_{2}}$ in $\Gamma_{\alpha, \beta}$ for $\beta $ sufficiently large.
\end{theorem}  Our result  states that the free convolution operation may be linearized through the  $\mathcal{R}$-transform, provided that the algebra $\mathcal{B}$ is finite dimensional (counterexamples are provided if this condition is relaxed).  The bounded, scalar valued case was first proven in \cite{Vadd}.  The bounded, operator valued version was proved in \cite{V2}.  The unbounded, scalar-valued version of this result first appeared in \cite{BV}, although it was proven with a finite variance assumption in \cite{Maass}.

This paper is organized as follows.
 Section \eqref{Prelim} includes introductory information and preliminary lemmas for later use.
In section \eqref{OpIso}, we prove Theorem \eqref{unboundiso}, which states that the Cauchy transform associated to an affiliated operator $a$ encodes all of the information
on the von Neumann algebra $(\mathcal{B} \cup \{a\})''$, provided that it's values are known on a (surprisingly large) subset of the resolvent.
In section \eqref{LimitingO}, we prove Theorem \eqref{MainTheorem}, which shows that the Cauchy transforms of unbounded operators may be characterized completely as the limiting objects of the Cauchy transforms of bounded operators in an appropriate function theoretic topology.
Section \eqref{CounterE} includes an example, arising from matrices whose entries are Cauchy distributions, wherein the Cauchy transform encodes very little of the operator algebraic information on $(\mathcal{B} \cup \{a\})''$, if  only the restriction to the non-commutative upper-half plane is known.  This allows us to conclude that Theorem \eqref{unboundiso} is an optimal result.  Section \eqref{RTransform1} defines the $\mathcal{R}$-transform for finite dimensional $\mathcal{B}$, shows that random variables with tight scalar distribution functions have $R$-transforms with a common domain.  We also prove that the free convolution operation is linearized by these functions.  Section \eqref{RTransform2} provides an example wherein the $\mathcal{R}$-transform need not exist in any meaningful sense for $\mathcal{B}$ infinite dimensional.

\begin{acknowledgements}
Part of the computations found in section \eqref{CounterE} were done during joint investigations of the Cauchy distribution with Yoann Dabrowski in Lyon.  I am grateful to Yoann for the accommodation.  The author also worked on this project during a research visit to Universit\'e Toulouse III, and is grateful towards Serban Belinschi and Charles Bordencave for their hospitality.  I am grateful towards the referee for excellent advice and insights.  Lastly, the author is funded through the Alexander von Humboldt Stiftung and gratefully acknowledges their support.
\end{acknowledgements}

\section{Preliminaries}\label{Prelim}
\subsection{Free Probabilistic Preliminaries}

We refer to \cite{Speichop} for an overview of operator-valued free probability.

Let $\mathcal{A}$ denote a C$^{\ast}$-algebra and $\mathcal{B} \subset \mathcal{A}$ a unital $\ast$-subalgebra, complete with a conditional expectation $E: \mathcal{A} \mapsto \mathcal{B}$ (that is, $E$ is unital, bimodular, completely positive map).
We refer to the triple $(\mathcal{A}, E, \mathcal{B})$ as an \textit{operator-valued probability space}.

Let $\mathcal{B}$ be algebraically free from a self adjoint symbol $X$ and denote the $\ast$-algebra that they generate as $\mathcal{B}\langle X \rangle$,
the \textit{ring of non-commutative polynomials}.
Let  $a \in \mathcal{A}^{sa}$ and define the \textit{B-valued  distribution } of $a$  to be the map
$$ \mu_{a} : \mathcal{B}\langle X \rangle \mapsto \mathcal{B} \ ; \ \ \mu_{a}(P(X)):= E(P(a)). $$

We define $\Sigma_{0}$ as the set of all maps $\mu: \mathcal{B}\langle X \rangle \mapsto \mathcal{B}$ such that
\begin{enumerate}
\item For all $b,b'\in \mathcal{B}$ and $P(X) \in \mathcal{B}\langle X\rangle$, we have that $\mu(bP(X)b') = b\mu(P(X))b'$.
\item For $P_{1}(X), P_{2}(X), \ldots , P_{k}(X) \in \mathcal{B}\langle X \rangle$ we have that $$\left[ \mu(P_{j}^{\ast}(X)P_{i}(X) \right]_{i,j=1}^{k} \geq 0.$$
\item\label{CBND} There exists an $M > 0$ such that for any $b_{1} , b_{2} , \ldots , b_{\ell} \in \mathcal{B}$, we have that $$\mu(b_{1}Xb_{2} \cdots Xb_{\ell}) \leq M^{\ell - 1} \| b_{1}\| \|b_{2}\| \cdots \|b_{\ell}\|.$$
\end{enumerate}
It was shown in \cite{PV} that $\mu \in \Sigma_{0}$ if and only if there is an operator valued probability space so that $\mu$ arises as the distribution of an element $a \in \mathcal{A}$.  The set of distributions so that \eqref{CBND} holds for a fixed $M > 0$  shall be referred to as
$\Sigma_{0,M}$.   The following proposition states that this set is compact in the \textit{pointwise weak topology}.

\begin{proposition}[\cite{WilJFA}]\label{compactnessBS}
Assume that $\mathcal{B}$ is a W$^{\ast}$-algebra.  The space $\Sigma_{0,M}$ is compact in the sense that for any  sequence $\{ \mu_{n}\}_{n \in \mathbb{N}} \subset \Sigma_{0,M}$ there exists a $\mu \in \Sigma_{0,M}$ and an increasing sequence $\{ i_{n} \}_{n \in \mathbb{N}} \subset \mathbb{N}$ such that, for any $P(X) \in \mathcal{B}\langle X \rangle$ we have that
$$ \mu_{i_{n}}(P(X)) \rightarrow \mu(P(X)) $$
in the weak topology.
\end{proposition}

Given elements  $\{ a_{i} \}_{i \in I} \subset \mathcal{A}$, we say that these elements are $\mathcal{B}$-free if
\begin{equation}\label{freeness} E(P_{1}(a_{i_{1}})  P_{2}(a_{i_{2}}) \cdots P_{k}(a_{i_{k}})) = 0 \end{equation}
whenever $$E( P_{j}(a_{i_{j}}) ) = 0$$
for $j= 1, 2 , \ldots , k$ and $i_{1} \neq i_{2} , i_{2} \neq i_{3} , \ldots , i_{k-1} \neq i_{k}$ (we will hence force refer to these as \textit{alternating products}).  Given a family of $\ast$-subalgebras, $\{A_{i} \}_{i \in I} \subset \mathcal{A}$ we similarly define freeness as the property that alternating products of centered elements have expectation $0$. The property in equation \eqref{freeness} is equivalent to
the $\ast$-subalgebras generated by $\{ a_{i} , \mathcal{B} \} $ being $\mathcal{B}$-free.
If the base algebra $\mathcal{B}$ is simply $\mathbb{C}$, we say that these elements are \textit{free}.

The following is well know and could, according to the viewpoint, be considered the fundamental theorem of operator-valued free probability, insofar as it shows that the study of joint distributions of free random variables and, by extension, random matrices, can be conducted using operator-valued machinery.

\begin{lemma}\label{bslemma}
Let $(\mathcal{A} , \tau)$ denote a C$^{\ast}$ probability space and $A_{i}  \subset \mathcal{A}$ free subalgebras for $i \in I$.
Consider the operator valued probability space $(M_{n}(\mathcal{A}), \tau \otimes 1_{n} , M_{n}(\mathbb{C}) )$ so that $\mathcal{B} = M_{n}(\mathbb{C})$.
Then the subalgebras $\{ M_{n}(\mathcal{A}_{i}) \}_{i \in I}$ 
are $\mathcal{B}$-free with respect to the conditional expectation $\tau \otimes 1_{n}$.
\end{lemma}
\begin{proof}
Let $A_{p} = (a_{\ell,m}^{(p)} )_{\ell, m =1}^{n}  \in M_{n}(\mathcal{A}_{i_{p}})$ for $p = 1 , \ldots , k$ .
Assume that \begin{equation}\label{bo}
\tau \otimes1_{n} (A_{p}) = 0
\end{equation}
 for $p = 1, \ldots , k$ and $i_{j} \neq i_{j + 1}$ for $j = 1, \ldots , k-1$. 
Note that $\eqref{bo}$ implies that
$$ \tau(a_{\ell,m}^{(p)}) = 0 $$ for all $p = 1, \ldots , k$ and $\ell , m = 1, \ldots , n$.
Restricting to a single entry, note that
$$(A_{1} A_{2} \cdots A_{k})_{r,s} = \sum_{m_{1} , \ldots , m_{k-1}=1}^{n} a^{(1)}_{r, m_{1}} a^{(2)}_{m_{1}, m_{2}} \cdots a^{(k)}_{m_{k-1}, s}$$
However, $$ \tau\left( a^{(1)}_{r, m_{1}} a^{(2)}_{m_{1}, m_{2}} \cdots a^{(k)}_{m_{k-1}, s} \right) = 0$$
since this is an alternating product of trace $0$ elements that are free.  We conclude that
$$ \tau \otimes 1_{n} (A_{1} A_{2} \cdots A_{k}) = 0_{n} $$
proving $\mathcal{B}$-freeness.
\end{proof}

Consider $\mathcal{B}$-free elements  $a_{1} , a_{2} \in \mathcal{B}$ with distributions $\mu_{1}$ and $\mu_{2}$, define the \textit{free convolution} of these operators, $\mu_{1} \boxplus \mu_{2}$, to be the distribution of the element $a_{1} + a_{2}$.

We study these objects through their function theory.  Indeed, let $a \in \mathcal{A}^{sa}$.
For every $n$ let $ M^{+}_{n}(\mathcal{B})$ denote the set of $b \in M_{n}(\mathcal{B})$ such that there exists an $\epsilon > 0$ so that $\Im(b) > \epsilon 1_{n}$.
Please note that this is hardly universal notation as $\mathcal{A}^{+}$ often denotes the positive cone of an operator algebra.

We define the \textit{Cauchy transform} of $a$ to be the non-commutative function
$$G_{a} = \{ G^{(n)}_{a} \} : \sqcup_{n=1}^{\infty} M_{n}^{+}(\mathcal{B}) \mapsto  \sqcup_{n=1}^{\infty} M_{n}^{-}(\mathcal{B})    \ ; \ \ G^{(n)}(b) := E[(b-a\otimes1_{n})^{-1}] $$
Note that this satisfies the properties that, for every $b \in M_{n}^{+}(\mathcal{B})$, $b' \in M_{n}^{+}(\mathcal{B})$ and
$S \in M_{n}(\mathbb{C})$, we have
$$ G^{(n+m)}(b \oplus b') = G^{(n)}(b ) \oplus G^{(m)}( b') \ ; \ \ G^{(n)}(SbS^{-1}) = SG^{(n)}(b)S^{-1}.$$
Thus, this is an example of a \textit{non-commutative function} and we refer to \cite{KVV} for an introduction to this theory. Moreover, this is an analytic function and we refer to \cite{Wil6} for an introduction to the analytic aspects of these functions.
We note here that this function extends to the entire resolvent, a fact that will become important in the coming sections.

The following is well known and simply states that moments may be recovered from the Cauchy transform.
\begin{lemma}\label{momentlemma}
Let $X_{1} , X_{2} \in \mathcal{A}$ .  Then the joint, $\mathcal{B}$-valued distribution of $X_{1}$ and $X_{2}$ can be recovered from the non-commutative function
$$ g^{(n)}(B) :=  B \in M^{+}_{2n}(\mathcal{B}) \mapsto E_{2n}\left[ \left( B - \left( \begin{array}  {cc}
  X_{1} & 0  \\
0  & X_{2}
   \end{array} \right) \otimes 1_{n} \right)^{-1} \right]. $$
\end{lemma}

We define the $\mathcal{R}$-transform to be the function
$$ \mathcal{R}_{a}^{(n)}(b):= (G^{(n)}_{a})^{\langle -1 \rangle}(b) - b^{-1} $$
where the $\langle-1\rangle$ superscript refers to the composition inverse.  In this setting, these functions are always defined
for $\|b^{-1}\|$ small enough, whereas the domain of this function will be much more delicate when we begin studying unbounded operators.  The key property of these functions is that, for $\mathcal{B}$-free random variables $a_{1} , a_{2} \in \mathcal{A}^{sa}$,
$$ \mathcal{R}_{a_{1} + a_{2}} = \mathcal{R}_{a_{1} } + \mathcal{R}_{ a_{2}} .$$

\subsection{Operator Algebraic Preliminaries}

We refer to \cite{KRe, KR} for an introduction to the theory of operator algebras.
Let $\mathcal{A}$ denote a von Neumann algebra with trace $\tau$.  We say that an closed, unbounded operator $X$ is \textit{affiliated} with $\mathcal{A}$ if, for every $U \in \mathcal{A}'$, 
we have that $UX= XU$, with agreement of domains.
The theory of affiliated operators was initiated in \cite{MVN} where, crucially, it is was shown that the set affiliated operators  (which naturally contains $\mathcal{A}$ itself), denoted $\tilde{\mathcal{A}}$, forms a $\ast$-algebra.

Thus, \textbf{for the remainder of the paper}, we will assume that our $\mathcal{B}$-valued probability space $(\mathcal{A}, E, \mathcal{B})$
has the property that $\mathcal{B} \subset \mathcal{A}$ is a containment or W$^{\ast}$-algebras and there exists a trace $$ \tau: \mathcal{A} \mapsto \mathbb{C}$$
so that  the conditional expectation $E$ is $\tau$ preserving.

Given a normal operator $X \in \tilde{\mathcal{A}}$ we refer to $e_{X}$ as \textit{the spectral measure}  associated to $X$.
The normal affiliated operators may be alternatively  characterized by the property that   $e_{X}(\omega) \in \mathcal{A}$
for all Borel sets $\omega \subset \mathbb{C}$.

The following is probably well known but the author cannot find a reference.
\begin{lemma}\label{affiliation}
Assume that $T \in \mathcal{A}^{sa}$ satisfies $\ker{(T}) = \{0\}$.  Then $T^{-1} \in \tilde{\mathcal{A}}^{sa}$.
\end{lemma}
\begin{proof}
Let $\mathcal{H} = L^{2}(\mathcal{A}, \tau)$.  Define $$ T^{-1} : T\mathcal{H} \mapsto \mathcal{H} \ , \ \ T^{-1}T\xi := \xi. $$
By Proposition 2.5.13 in \cite{KRe}, $T^{-1}$ is a densely defined operator.
We show that it is closed.

Consider a convergent sequence in the graph $$( \eta_{n}, T^{-1}\eta_{n} ) = (T\xi_{n} , \xi_{n})  \rightarrow \{ \eta, \omega \} \in \mathcal{H} \times \mathcal{H}$$
where $\eta_{n} = T\xi_{n}$ for some $\xi_{n} \in \mathcal{H}$.
In particular, $\xi_{n} \rightarrow \omega$ so, by continuity of $T$, $T\xi_{n} \rightarrow T\omega$ and we may conclude that $\eta = T\omega$.  Thus, $\eta$ is in the domain of $T^{-1}$ and $T^{-1}\eta = \omega$.
This proves that $T^{-1}$ is a closed operator.

To show that $T^{-1} \in \tilde{\mathcal{A}}$, let $U \in \mathcal{A}'$.
For $T\xi$ in the domain of $T^{-1}$, note that $$UT^{-1}U^{\ast}T\xi  = UT^{-1}TU^{\ast}\xi = UU^{\ast}\xi = \xi = T^{-1}T\xi$$
proving affiliation.

Note  that $T^{-1}$ is a symmetric operator.  Indeed, for $x = T\xi$ and $y = T\eta$ in the domain of $T^{-1}$, we have that
$$ \langle T^{-1}x , y \rangle  = \langle \xi , T\eta \rangle = \langle T\xi  , \eta \rangle = \langle x , T^{-1}y \rangle$$
where we utilized self-adjointness of $T$.

By Proposition 2.7.10 in \cite{KRe},  self-adjointness of $T^{-1}$  will follow if we can show that $T^{-1} \pm i1$ has dense range.
Let $\eta \in \mathcal{H}$ and define $$\xi = -iT(T - i)^{-1}\eta.$$
Note that $\xi$ is in the domain of $T^{-1}$ since $i$ is in the resolvent of $T$ by self-adjointness.
Moreover, as the affiliated operators form an algebra, we have that
$$ $$
$$ (T^{-1} + i) \xi = (T^{-1} + i)[-iT(T- i)^{-1}]\eta = [(T - i ) (iT^{-1})][-iT(T- i)^{-1}]\eta = \eta.$$
This proves density of the range.  Dense range of $(T^{-1} - i) $ follows similarly.
\end{proof}

\subsection{Resolvent Sets}

Let $a \in \tilde{\mathcal{A}}$.  The \textit{resolvent set} of $a$ with respect to a subalgebra $\mathcal{B}$ is defined as the set of $b \in \mathcal{B}$ such that $(a - b)^{-1} \in \mathcal{A}$.

\begin{lemma}\label{affresolvent}
Assume that $a \in \tilde{ \mathcal{A}}^{sa}$ and $b \in \mathcal{A}^{sa}$ satisfies $|b| > \epsilon1$ for some $\epsilon > 0$.
Then $(a + ib)^{-1} \in \mathcal{A}$ and
 \begin{equation}\label{inverseineq1}
\| (a + ib)^{-1} \| \leq 1/\epsilon.
\end{equation}
Further, we have that $a(a + ib)^{-1} \in \mathcal{A}$ and satisfies the inequality
 \begin{equation}\label{inverseineq2}
\|a(a + ib)^{-1} \| \leq 1 + \|b\|/\epsilon.
\end{equation}
Moreover, if $b > \epsilon1$, then we may also conclude that $\Im{( (a + ib)^{-1})} <0$.
\end{lemma}
\begin{proof}

We first show that $(a + ib)^{-1} \in \mathcal{A}$ and satisfies \eqref{inverseineq1}.
Let $e_{a}$ denote the spectral measure associated to $a$ and define $$p_{N} := e_{a}([-N,N]) \ ; \ \ a_{N} := ap_{N}.$$
Letting $\sigma(\cdot)$ denote the spectrum of an operator, we have that  $$\sigma(a_{N} + ib) \subset \sigma(a_{N}) + i\sigma(b) \subset [-N,N] \times i \left( \mathbb{R}\setminus[-\epsilon,\epsilon]\right)$$ (it is unclear if an analogous fact is true for unbounded operators, hence the proof).
In particular, $(a_{N} + ib)^{-1} \in \mathcal{A}$.
Using the spectral mapping theorem, we have that 
$$ \sigma( (a_{N} + ib)^{-1}) = \{ \lambda^{-1} : \lambda \in \sigma( a_{N} + ib) \} \subset B_{1/\epsilon}(\{ 0 \}).$$
We may assume that $(a_{N} + ib)^{-1}$ converges in the weak operator topology to a cluster point $S \in \mathcal{A}$ with $\| S\| < 1/\epsilon$.  The norm bound follows from agreement of the norm and spectral radius in C$^{\ast}$-algebras.

Now, $\tilde{\mathcal{A}}$ is an algebra so that it contains $S(a+ib)$.  Taking the standard representation on $L^{2}(\mathcal{A},\tau)$,
consider the dense subset $\{ p_{N} L^{2}(\mathcal{A},\tau) \}_{N \in \mathbb{N}}$, and note that these are in the domain of the product $S(a+ib)$.
For $\xi, \eta \in L^{2}(\mathcal{A})$, we have that, for any $M \geq N$,
\begin{equation}\label{Sequal}
  (a+ ib) p_{N}\xi  = (a_{M}+ ib) p_{N}\xi 
\end{equation}
Thus, \begin{align*} \langle S(a+ ib) p_{N}\xi , \eta \rangle & = \lim_{M \uparrow \infty}  \langle (a_{M} + ib)^{-1}(a+ ib) p_{N}\xi , \eta \rangle \\
& = \lim_{M \uparrow \infty}  \langle (a_{M} + ib)^{-1}(a_{M}+ ib) p_{N}\xi , \eta \rangle  \\ 
& = \langle p_{N}\xi , \eta \rangle.
\end{align*}
and we may conclude that
\begin{equation}
 S(a+ ib)p_{N}\xi = p_{N}\xi
\end{equation}
for all $N \in \mathbb{N}$ and $\xi \in L^{2}(\mathcal{A})$.

Now, let $(x,y)$ denote any element in the graph of $S(a+ ib)$.  Observe that
$(p_{N}x,p_{N}x)$ is in the graph of this operator and converges to $(x,x)$ in the $L^{2}$ norm on the product.  We conclude that $x=y$.
Thus, $S(a+ ib) = 1$.  By similar methods, one may show that $$ \langle (a+ib)S\eta , P_{N}\xi \rangle = \langle \eta , P_{N}\xi \rangle $$
for all $\eta \in \mathcal{D}( (a+ib)S)$ , $\xi \in L^{2}(\mathcal{A})$ and $N \in \mathbb{N}$.  By the same method, we conclude that
$(a+ib)S = 1$ and our result follows.


Regarding the second part of the statement, note that since, $\tilde{\mathcal{A}}$ is an algebra, we have that
$a(a + ib)^{-1} \in \tilde{\mathcal{A}}$.  Moreover,
$$ a(a + ib)^{-1} = (a + ib)(a + ib)^{-1} - ib(a + ib)^{-1}= 1 - ib(a + ib)^{-1}$$
which implies that $a(a + ib)^{-1} \in \mathcal{A}$ and maintains bound \eqref{inverseineq2}.

For $b > \epsilon1$, 
we have that
$$ (a + ib)^{-1} = b^{-1/2}(i + b^{-1/2}ab^{-1/2})^{-1}b^{-1/2}. $$
The last statement follows by applying the spectral theorem for the normal operator $ b^{-1/2}ab^{-1/2}$.
\end{proof}

As a (well known) corollary, we have that for $a \in \tilde{\mathcal{A}}^{sa}$, 
we have that the resolvent set of $a$ contains $\mathcal{B}^{+} \sqcup \mathcal{B}^{-}$.

\begin{lemma}\label{expandedresdef}
Let  $b_{1} , b_{3} , \ldots , b_{2n-1} \in \mathcal{B}^{+}$, $b_{2} , b_{4} , \ldots , b_{2n} \in \mathcal{B}^{-}$
and define  $r >  0$ to be the largest constant such that $$ r \leq |\Im{(b_{i})}|$$
for $i = 1, 2, \ldots , 2n$.
Then the open ball of radius $r$ about the point
$$ b = \left[\begin{array}  {ccccc}
  b_{1} & 0  & 0 & \cdots & 0 \\
0 & b_{2}  & 0 & \cdots & 0 \\
0 & 0  & b_{3} & \cdots & 0 \\
\ & \vdots & \ & \vdots & \ \\
0 & 0  & 0 & \cdots & b_{2n}
   \end{array} \right]   $$
is contained in the resolvent set of $a\otimes 1_{2n}$ for all
 $a \in \tilde{\mathcal{A}}^{sa}$.
\end{lemma}
\begin{proof}
Let $b' \in M_{2n}(\mathcal{B})$ satisfy $\|b'\| < r$.
Then, \begin{equation}\label{elmo} a - (b + b') = (a-b)[1 - (a-b)^{-1}b')] \end{equation}
where $(a-b)$ is invertible by entry-wise application of \eqref{affresolvent}.
The right hand side of \eqref{elmo} is a product of invertible elements since \eqref{affresolvent} implies that $$\|  (a-b)^{-1}b')\| < 1.$$
This completes the proof.
\end{proof}

\begin{remark}\label{setremark}
Define $\Omega_{n}(\mathcal{B})$ as the union of all $B_{r}(b)$ where $r \in \mathbb{R}_{+}$ and $b \in M_{2n}(\mathcal{B})$ come from Lemma \eqref{expandedresdef}.
We define 
 $$ \Omega(\mathcal{B}) := \sqcup_{n=1}^{\infty} \Omega_{n}(\mathcal{B}). $$
We note that each $\Omega_{n}(\mathcal{B})$ is a connected, open set, so that an analytic function
defined on $\Omega_{n}(\mathcal{B})$ is characterized by its value on any open subset.  Note that this is not a non-commutative set but that each $X\in \Omega(\mathcal{B})$ has a non-commutative neighborhood in this set.

We also remark that, since $\Omega_{n}(\mathcal{B}) \subset M_{2n}(\mathcal{B})$, the non-commutative Cauchy transforms satisfy $$ G_{X}^{(2n)}: \Omega_{n}(\mathcal{B}) \mapsto M_{2n}(\mathcal{B}), $$
where we are only analyzing $G_{X}$ in even dimensions.  However, this contains all of the information on the function since the Cauchy transform respects direct sums.
\end{remark}

\subsection{Analytic Miscellani}
The following are various classical analytic results that will bolster our toolbox in the coming sections.

We cite an extension of the Bloch Theorem for  Banach spaces.
\begin{theorem}\cite[Theorem 1]{Balls}\label{Bloch}
Let $h: B_{R} \mapsto B_{M}$ be a holomorphic function such that $Dh(0)^{-1}$ exists.  Suppose that $\|Dh(0)^{-1}\|^{-1} \geq a$  and put
$$ r = \frac{R^{2}a}{4M}, \ \ \ P = \frac{R^{2}a^{2}}{8M} $$
Then $h$ maps $B_{r}$ biholomorphically onto a domain covering $B_{p}(h(0))$.
\end{theorem}

Theorem 3.17.17 in \cite{Hille} provides Lipschitz type estimates for vector valued analytic functions.
Indeed, for an analytic function $f$ that is locally bounded by $M(a)$ in a neighborhood of radius $r_{a}$, we have that
\begin{equation}\label{lipschitz}
\| f(y) - f(x) \| \leq \frac{2M(a) \| x - y \|}{r_{a} - 2 \| x - y \|}
\end{equation}

We extend Lemma 3.2 in \cite{WilJFA}.

\begin{proposition}\label{GConvExt}
Assume that $\mu_{k} , \mu \in \Sigma_{0,M}$ for all $k \in \mathbb{N}$.  Let $a , a_{k} \in \mathcal{A}$
denote elements with distributions $\mu$ and $\mu_{k}$. Assume that $$ \mu_{k} \rightarrow \mu $$
in the pointwise weak topology.
Then $$ G_{\mu_{k}}^{(2n)} \rightarrow G_{\mu}^{(2n)}  $$
pointwise  in the weak topology on $\Omega_{n}(\mathcal{B})$.

Further, assume that $b \in M_{n}(\mathcal{B})^{sa}$ is contained in the resolvent of $a$ and $a_{k}$  for $k \geq K$ and has the property that,
for some $\delta > 0$, there exists a $C > 0$ such that
$$ \| (a - b')^{-1} \| < C \ ; \ \  \| (a_{k} - b')^{-1} \| < C $$
for all $b' \in B_{\delta}(b)$ and $k \geq K$.
Then $$G^{(n)}_{\mu_{k}}(b) \rightarrow G^{(n)}_{\mu}(b)$$
 in the weak topology.
\end{proposition}
\begin{proof}
Let $\{ \phi_{i} \}_{i=1}^{N} \subset M_{2n}(\mathcal{B})^{\ast}$.
Pick $b \in \Omega_{n}$.  By \eqref{setremark},
$$ b =  \left[\begin{array}  {ccccc}
  b_{1} & 0  & 0 & \cdots & 0 \\
0 & b_{2}  & 0 & \cdots & 0 \\
0 & 0  & b_{3} & \cdots & 0 \\
\ & \vdots & \ & \vdots & \ \\
0 & 0  & 0 & \cdots & b_{2n}
   \end{array} \right] + T  $$
where $b_{2i} \in \mathcal{B}^{-}$ and $b_{2i - 1} \in \mathcal{B}^{+}$ for $i = 1 , 2 , \cdots , n$ and
\begin{equation}\label{forref}
\|T\| < |\Im{(b_{j})}|
\end{equation}
 for $j = 1, 2 , \ldots , 2n$.
This implies that
\begin{equation}
\left[\begin{array}  {ccccc}
  b_{1} + i\lambda & 0  & 0 & \cdots & 0 \\
0 & b_{2} - i\lambda  & 0 & \cdots & 0 \\
0 & 0  & b_{3} + i\lambda & \cdots & 0 \\
\ & \vdots & \ & \vdots & \ \\
0 & 0  & 0 & \cdots & b_{2n} - i\lambda
   \end{array} \right] + T = b + 
\left[\begin{array}  {cc}
 i\lambda  & 0 \\
0  & -i\lambda
  \end{array}\right]\otimes1_{n}
\end{equation}
is an element of $\Omega_{n}(\mathcal{B})$ for all $\lambda \in \mathbb{R}_{+}$ since the right side of \eqref{forref} is increasing with $\lambda$.
Passing to a small neighborhood $\Gamma$ of $\{ i\lambda \}_{\lambda \in \mathbb{R}_{+}}$, openness implies that
$$  b + 
\left[\begin{array}  {cc}
z  & 0 \\
0  & -z
  \end{array}\right]\otimes1_{n} \in \Omega_{n}(\mathcal{B})$$
for all $z\in \Gamma$.

We consider the family of complex functions
$$ \left\{ \phi_{i} \circ G^{(2n)}_{\mu_{k}} \left(  b + 
\left[\begin{array}  {cc}
z  & 0 \\
0  & -z
  \end{array}\right]\otimes1_{n} \right) \right\}_{k \in \mathbb{N} \ , \ i=1, 2, \ldots , N}: \Gamma \mapsto \mathbb{C}. $$
The Cauchy transforms are uniformly bounded on this set by Lemma \eqref{affresolvent} and the 
fact that distributions are in $\Sigma_{0,M}$ implies that $$\|a_{k}\| , \|a\| \leq M.$$
This implies that we have a normal family of complex functions.  By Lemma 3.2 in \cite{WilJFA}, 
we have convergence to $\phi_{i} \circ G_{\mu}$ for those $z \in \Gamma$ with $\Im{(z)}$ large enough.
By normality, this convergence must hold on all of $\Gamma$, in particular when $z=0$.  As this family of functionals was arbitrary, we have weak convergence.


Shifting to $b \in M_{n}(\mathcal{B})^{sa}$ , we redefine $$\Gamma := B_{\delta}(\{0 \}) \cup \{z :\  \Im{(z)} > \delta/2 \} \subset \mathbb{C} \  ; \ \ \Lambda:= \{ b + z1_{n} \}_{z \in \Gamma}$$
By assumption and Lemma \eqref{affresolvent}, there exists a $C' > 0$ such that
$$ \| (a - b')^{-1} \| < C' \ ; \ \ \| (a_{k} - b')^{-1} \| < C' $$
for all $b' \in \Lambda$.  This implies that $$\left\{ \phi_{i} \circ  G^{(n)}_{\mu_{k}}\left( b + z1_{n}   \right) \right\}_{i=1 , \ldots , N \ ; \ k \geq K} : \Gamma \mapsto \mathbb{C}$$
is a normal family of complex analytic functions.  The same argument proves weak convergence.
\end{proof}



\section{Convolution Operations}\label{OpIso}
The main result of this section is the extension of the various convolution operations to unbounded operators.  This will follow from  the more general Theorem \eqref{BigTheorem}.  We must first show that the Cauchy transform encodes all of the information of the distribution associated to $a$.

The following is   well known, though we cannot find a good reference for this specific variation ( see \cite{NS}, Theorem 4.11 for a   variation with $\mathcal{B}$ finite dimensional).
\begin{lemma}\label{isomorphicalg}
Let $a_{1} , a_{2} \in \mathcal{A}$ denote two random variables with the same $\mathcal{B}$-valued $\ast$-distributions.
Then the tracial W$^{\ast}$-algebras $\mathcal{A}_{1} = \{ a_{1}, a_{1}^{\ast} , \mathcal{B} \}''$ and    $ \mathcal{A}_{2} = \{ a_{2} , a_{2}^{\ast}, \mathcal{B} \}''$ are spatially isomorphic and this isomorphism commutes with the conditional expectation.
\end{lemma}
\begin{proof}
Let $\mathcal{H}_{i}$ denote the Hilbert space generated by $\mathcal{A}_{i}$ with inner product
$$\langle \psi, \eta \rangle_{i} = \tau(\eta^{\ast} \psi).$$ In each case, consider the dense family $\{P(a_{i}, a_{i}^{\ast}) \}_{P(X, X^{\ast}) \in \mathcal{B}\langle X , X^{\ast} \rangle}$.
Observe that \begin{align*}
\langle P(a_{1}, a_{1}^{\ast}) , Q(a_{1}, a_{1}^{\ast}) \rangle _{1} & = \tau(Q^{\ast}(a_{1}, a_{1}^{\ast})P(a_{1}, a_{1}^{\ast})  ) = \tau(E[Q^{\ast}(a_{1}, a_{1}^{\ast})P(a_{1}, a_{1}^{\ast})]  ) \\
& = \tau(E[Q^{\ast}(a_{2}, a_{2}^{\ast})P(a_{2}, a_{2}^{\ast})] ) =  \langle P(a_{2}, a_{2}^{\ast}) , Q(a_{2}, a_{2}^{\ast}) \rangle_{2}.
\end{align*}
Thus, the map sending $a_{1}$ to $a_{2}$ and $a_{1}^{\ast}$ to $a_{2}^{\ast}$ extends to a Hilbert space isometry.
As this intertwines the actions of the von Neumann algebras, this defines a spatial isomorphism of W$^{\ast}$-algebras.
\end{proof}


\begin{remark}
Let $a \in \tilde{\mathcal{A}}^{sa}$.  Observe that the Cauchy transform
$$ G^{(n)}(b) := E[(b - a\otimes1_{n})^{-1}] $$
is well defined on the entire resolvent set of the operator $a$.  
Typically, we consider the domain of this function to be the distinguished subset $M_{n}^{+}(\mathcal{B})$.  However, in the unbounded case, the resolvent may be very disconnected.  As we will show in section \eqref{CounterE}, it is possible for two Cauchy transforms to agree on the non-commutative upper-half plane and still have distinct extensions to the entire resolvent.  Thus, the domain $\Omega(\mathcal{B})$ from remark \eqref{setremark} is required for characterization of these operators through the function theory.  We may conclude that the following is an optimal result.
\end{remark}

\begin{theorem}\label{unboundiso}
Let $a_{1} , a_{2} \in \tilde{\mathcal{A}}^{sa}$ denote two random variables whose Cauchy transforms agree on the set $\Omega$.
Then the tracial W$^{\ast}$-algebras $\mathcal{A}_{1} = \{ a_{1} , \mathcal{B} \}''$ and    $ \mathcal{A}_{2} = \{ a_{2} , \mathcal{B} \}''$ are spatially isomorphic.
\end{theorem}
\begin{proof}

Note that 
$$\mathcal{A}_{j} = \{ a_{j} , \mathcal{B} \}'' = \{(a_{j} + i)^{-1}, (a_{j} - i)^{-1} , \mathcal{B} \}'' $$ (both containments follow from the spectral theorem for normal operators and standard manipulations).

Thus, we consider the operators
$$  T_{j} = \left[\begin{array}  {cc}
   (a_{j} - i)^{-1}  & 0 \\
0 & (a_{j} + i)^{-1} 
  \end{array} \right]$$
$$ S_{ j} =   \left[\begin{array}  {cc}
   0_{2n} & T_{j}\otimes1_{n} \\
T_{j}^{\ast} \otimes 1_{n} & 0_{2n}
   \end{array}\right] $$

We have that, for $B \in M_{2n}(\mathcal{B})$,
\begin{align} \label{lhs}
& \left( \left[\begin{array}  {cc}
   0_{2n} & B \\
B^{\ast}  & 0_{2n}
   \end{array}\right]  - S_{j} \right) ^{-1}  =  \left[\begin{array}  {cc}
   0_{2n} &  (B^{\ast} - T_{j}^{\ast}\otimes 1_{n})^{-1} \\
 (B - T_{j}\otimes1_{n})^{-1}    & 0_{2n}
   \end{array}\right] \\
&= \nonumber \left[\begin{array}  {cc}
   0_{2n} & [(T_{j}^{-1} \otimes 1_{n})[(T_{j}^{-1} \otimes 1_{n}) - B^{-1}]^{-1}B^{-1}]^{\ast}  \\
 (T_{j}^{-1} \otimes 1_{n})[(T_{j}^{-1} \otimes 1_{n}) - B^{-1}]^{-1}B^{-1}& 0_{2n}
   \end{array}\right] \\
&= \nonumber   \left[\begin{array}  {cc}
   0_{2n} & [B^{-1} +  B^{-1}[(T_{j}^{-1} \otimes 1_{n}) - B^{-1}]^{-1}B^{-1}]^{\ast}  \\
B^{-1} + B^{-1}[(T_{j}^{-1} \otimes 1_{n}) - B^{-1}]^{-1}B^{-1}& 0_{2n}
   \end{array}\right] \\
&=  \label{rhs} \left[\begin{array}  {cc}
   0_{2n} & [B^{-1} +  B^{-1}[(a_{j} \otimes 1_{2n}) - B_{0} -  B^{-1}]^{-1}B^{-1}]^{\ast}  \\
B^{-1} + B^{-1}[(a_{j} \otimes 1_{n}) -B_{0} - B^{-1}]^{-1}B^{-1}& 0_{2n}
   \end{array}\right] 
\end{align}
where
$$B_{0} =   \left[\begin{array}  {cc}
   i & 0\\
0 & -i 
   \end{array}\right]\otimes 1_{n} $$
Applying $E_{4n}$ to  \eqref{rhs}, we obtain the matrix
\begin{equation}\label{RHS}
 \left[\begin{array}  {cc}
   0 & [B^{-1} - B^{-1}G_{a_{j}}^{(2n)}(B_{0} + B^{-1})B^{-1}]^{\ast}\\
B^{-1} - B^{-1}G_{a_{j}}^{(2n)}(B_{0} + B^{-1})B^{-1} & 0
   \end{array}\right]
\end{equation}
provided that $\| B^{-1} \| < 1$.   The key point is that this function is independent of $j=1,2$.

 Moreover, applying $E_{4n}$ \eqref{lhs}, we obtain
\begin{equation}\label{LHS}
 G^{(4n)}_{S_{j}}\left( \left[\begin{array}  {cc}
   0_{2n} & B \\
B^{\ast}  & 0_{2n}
   \end{array}\right]  \right) = \left[\begin{array}  {cc}
   0_{2n} & E_{2n}[ (B^{\ast} - T_{j}^{\ast}\otimes 1_{n})^{-1} ]\\
E_{2n}[(B - T_{j}\otimes1_{n})^{-1} ]  & 0_{2n}
   \end{array}\right]  
\end{equation}
and, since $\| S_{j} \| \leq 1$, we have that the domain of this function extends to those $B$ with  $\| B^{-1} \| < 1$.  

Thus, we may conclude that $$ E_{2n}[(B - T_{1}\otimes1_{n})^{-1} ] = E_{2n}[(B - T_{2}\otimes1_{n})^{-1} ].$$
By lemma \eqref{momentlemma}, we may recover the respective joint $\mathcal{B}$-valued distributions of $(a_{1} \pm i)^{-1}$ and $(a_{2} \pm i)^{-1}$ from these functions, so that these joint distributions must also agree.  By lemma \eqref{isomorphicalg}, we have that 
$$\mathcal{A}_{1} = \{(a_{1} + i)^{-1}, (a_{1} - i)^{-1} , \mathcal{B} \}'' \cong \{(a_{2} + i)^{-1}, (a_{2} - i)^{-1} , \mathcal{B} \}'' = \mathcal{A}_{2}$$
proving our theorem.

\end{proof}

\begin{theorem}\label{BigTheorem}
Let $Q(X_{1}, X_{2}, \ldots , X_{k}) \in \mathcal{B}\langle X_{1}, X_{2}, \ldots , X_{k} \rangle^{sa}$.  Assume that
$T_{1}, T_{2}, \ldots , T_{k} \in \tilde{\mathcal{A}}^{sa}$ are $\mathcal{B}$-free.  Then the Cauchy transform
$$G_{Q(T_{1}, T_{2}, \ldots , T_{k})}$$ depends only on the Cauchy transforms $G_{T_{i}}$ for $i = 1,2, \ldots , k$.
\end{theorem}
\begin{proof}
Let $T_{1}, T_{2}, \ldots , T_{k} \in \tilde{\mathcal{A}}^{sa}$  and $S_{1}, S_{2}, \ldots , S_{k} \in \tilde{\mathcal{A}}^{sa}$  denote two families of operators so that the elements $T_{i}$ are pairwise $\mathcal{B}$-free and the elements $S_{i}$ are pairwise $\mathcal{B}$-free for $i = 1, 2, \ldots , k$.  Assume that $G_{T_{i}} = G_{S_{i}}$ for $i = 1, 2, \ldots , k$.

Note that $\mathcal{B}$-freeness implies that 
\begin{equation}\label{Tequiv}
 \{T_{1}, T_{2}, \ldots , T_{k}, \mathcal{B} \}'' \cong  \{ T_{1} , \mathcal{B} \}'' \ast_{\mathcal{B}} \{ T_{2} , \mathcal{B} \}'' \ast_{\mathcal{B}} \cdots  \ast_{\mathcal{B}} \{ T_{k} , \mathcal{B} \}'' 
\end{equation}
\begin{equation}\label{Sequiv}
 \{S_{1}, S_{2}, \ldots , S_{k}, \mathcal{B} \}'' \cong  \{ S_{1} , \mathcal{B} \}'' \ast_{\mathcal{B}} \{ S_{2} , \mathcal{B} \}'' \ast_{\mathcal{B}} \cdots  \ast_{\mathcal{B}} \{ S_{k} , \mathcal{B} \}'' 
\end{equation}
As we saw in \eqref{unboundiso}, $\{ T_{i} , \mathcal{B} \}'' \cong \{S_{i} , \mathcal{B} \}''$ so that the right hand sides of \eqref{Tequiv} and \eqref{Sequiv} are isomorphic and this isomorphism intertwines the conditional expectation.
We conclude that
\begin{align*}
G^{(n)}_{Q(T_{1}, T_{2}, \ldots , T_{k})}(b) & = E_{n}[(b - Q(T_{1}, T_{2}, \ldots , T_{k})\otimes1_{n})^{-1} ] \\ &= E_{n}[(b - Q(S_{1}, S_{2}, \ldots , S_{k})\otimes 1_{n})^{-1} ] \\ &= G^{(n)}_{Q(S_{1}, S_{2}, \ldots , S_{k})}(b) 
\end{align*}
for all $b \in M_{n}(\mathcal{B})$ and $n \in \mathbb{N}$.  This proves our theorem.
\end{proof}

\section{The Cauchy Transform of an Affiliated Operator}\label{LimitingO}

The main result of this section is the following.

\begin{theorem}\label{MainTheorem}
Let $G = \{ G^{(n)} \}_{n=1}^{\infty}$ denote a non-commutative function with domain containing $\Omega(\mathcal{B})$ (resp. $H^{+}(\mathcal{B})$).
Then there exists a $\mathcal{B}$-valued probability space $(\mathcal{A}, E, \mathcal{B})$ and an element $a \in \tilde{ \mathcal{A}}^{sa}$ such that $$G^{(n)}(b) = E[(b-a\otimes 1_{n} )^{-1}] $$ if and only if there exist a sequence of elements $a_{k} \in \mathcal{A}$ with distributions $\mu_{k}$ such that the following conditions hold:
\begin{enumerate}[I.]
\item\label{convergence}  $G_{\mu_{k}} \rightarrow G$ in the weak topology uniformly on all sets  $ \mathcal{O} \subset \Omega_{n}(\mathcal{B})$  (resp. $H^{+}(\mathcal{B})$) satisfying the property that there exist $C , r > 0$ such that  for all $b \in \mathcal{O}$, $\|b \| < C$ and $|\Im{(b)}| > r 1_{n}$.
\item\label{tightness} For every $\epsilon > 0$, there exists an $N \in \mathbb{N}$ such that $\tau(e_{k}([-N,N]) )> 1- \epsilon$ for every $k \in \mathbb{N}$.
\end{enumerate}
\end{theorem}
\begin{proof}
Regarding necessity, assume that $$G^{(n)}(b) = ( E\otimes 1_{n})[(b-a\otimes 1_{n} )^{-1}] $$ for some $a \in \tilde{ \mathcal{A}}^{sa}$.
By Lemma \eqref{expandedresdef}, this non-commutative function has domain containing $\Omega(\mathcal{B})$.

Setting notation again, we let 
$$  p_{k} = e_{a}([-k,k]) \ ; \ \ a_{k} = ap_{k}$$
where $e_{a}$ is, again, the functional calculus associated to $a$.
These random variables satisfy condition \eqref{tightness} by definition.
We will show that this sequence satisfies \eqref{convergence}.

Let $\mathcal{O} \subset \Omega_{n}(\mathcal{B})$  be such that there exists $C , r > 0$ so that, for all $b \in \mathcal{O}$ ,
$\| b \| < C$ and $\Im{(b)} > r1$ (this will also show convergence on $H^{+}(\mathcal{B})$ as it imbeds in this set).
Fix a positive element $p \in \mathcal{B}$  with $\| p\|_{L^{2}} \leq 1$ and note that the weak topology on $\mathcal{B}$ is generated by convergence with respect to finite collections of vector states, $$\phi_{p}(b) := \tau(bp) = \langle b, p \rangle$$
where the inner product is with respect to the standard representation.
Setting notation, we refer to 
\begin{equation}
 a^{(n)} := a\otimes1_{n} \ ; \ \ a_{k}^{(n)} := a_{k}\otimes1_{n} \ ; \ \ \tau_{n}:= Tr_{n} \circ \tau \otimes 1_{n} \ ; \ \ E_{n} = E \otimes 1_{n}
\end{equation}
where $Tr_{n}$ is the normalized trace.

To show convergence with respect to such a vector state, observe that 
\begin{align}
\phi_{p} &\left( E_{n}[(b-a^{(n)})^{-1}]  - E_{n}[(b - a_{k}^{(n)}\otimes1_{n})^{-1}] \right)\\ 
& = \tau_{n}\left( E_{n}[(b-a^{(n)})^{-1} - (b - a_{k}^{(n)})^{-1}]p \right) \label{P1} \\ 
&= \tau_{n}\left(  E_{n}[ ( (b-a^{(n)})^{-1} - (b - a_{k}^{(n)})^{-1})p] \right) \label{P2} \\
&=  \tau_{n}\left( ((b-a^{(n)})^{-1} - (b - a_{k}^{(n)})^{-1})p \right) \label{P3} \\
& = \tau_{n}\left( (b-a^{(n)})^{-1}(a^{(n)}[(1-p_{k})])(b - a_{k}^{(n)})^{-1} p\right) \label{P4} \\
&= \tau_{n}\left((1-p_{k}) (b - a_{k}^{(n)})^{-1} p (b-a^{(n)})^{-1}a^{(n)} \right) \label{P5} \\
& \leq \tau_{n}(1-p_{k})^{1/2} \tau_{n}\left( (b - a_{k}^{(n)})^{-1} p (b-a^{(n)})^{-1}|a^{(n)}|^{2}(b^{\ast}-a^{(n)})^{-1}p(b^{\ast} - a_{k}^{(n)})^{-1}  \right)^{1/2} \label{P6}\\
& \leq \tau_{n}\left( 1-p_{k}\right)^{1/2} \| (b - a_{k}^{(n)})^{-1} p (b-a^{(n)})^{-1}|a^{(n)}|^{2}(b^{\ast}-a^{(n)})^{-1}p(b^{\ast} - a_{k}^{(n)})^{-1} \|^{1/2} \label{P7} \\
& \leq \tau_{n}\left( 1-p_{k}\right)^{1/2} \| (b - a_{k}^{(n)})^{-1}\| \| p\| \| (b-a^{(n)})^{-1}a^{(n)}\|  \label{P8} \\
&\leq \tau_{n}\left( 1-p_{k}\right)^{1/2}  \|p\| \frac{1 + C/r}{r} \label{P9}
\end{align} 
where \eqref{P2} follows from bimodularity, \eqref{P3} from trace preservation, \eqref{P5} from traciality, \eqref{P6} from Cauchy-Schwarz, \eqref{P7} from complete positivity of $\tau$ and \eqref{P9} from Lemma \eqref{affresolvent}.  Since \eqref{P9} converges to $0$ as $k \uparrow \infty$, convergence in the weak topology follows.

To prove sufficiency, for each $k \in \mathbb{N}$, consider the element
\begin{equation}\label{element} A_{k} = \left[ \begin{array}  {cc}
    0 & (a_{k} + i)^{-1} \\
  (a_{k} -  i)^{-1} & 0
   \end{array} \right]   \in M_{2}(\mathcal{A})^{sa}.\end{equation}

Note that the elements $A_{k}$ are uniformly bounded with norm $1$ by \eqref{affresolvent}.
As we saw in Proposition \eqref{compactnessBS}, this set is compact in the pointwise weak topology and we refer to the limit point  as $T$.
We note that we may have passed to a larger algebra, but that the trace and conditional expectation and subalgebra are still preserved (this is Corollary 4.6 in \cite{WilJFA}).
We will still refer to the algebra as $\mathcal{A}$ and our notation will ignore the fact that we may have passed to a subsequence.

Regarding the operator $T$, we note that $t_{1,1} = t_{2,2} = 0$ and $t_{1,2} = t_{2,1}^{\ast}$ is a normal operator.
Indeed,  observe that T is self adjoint implies that $t_{1,1} = t_{1,1}^{\ast}$ and $t_{1,2} = t_{2,1}^{\ast}$.  The moments of $t_{1,1}$  satisfy
$$ 0 = \tau_{2}(E_{2}[ (e_{1,1} A_{k} e_{1,1})^{n} ]) \rightarrow \tau_{2}(E_{2}((e_{1,1}T e_{1,1})^{n})) = \tau((t_{1,1})^{n}), $$
proving our first claim.  Moreover, self-adjointness implies that $TT^{\ast} = T^{\ast}T$ and this implies normality of $t_{1,2}$.
Lastly, note that a similar argument implies that for an element $P(X,X^{\ast}) \in \mathcal{B}\langle X, X^{\ast}\rangle$, we have
\begin{equation}\label{jointmom}
\lim_{k\uparrow \infty} P((a_{k} + i)^{-1}, (a_{k} - i)^{-1}) = P(t_{1,2}, t_{2,1}) 
\end{equation}
where the convergence in the weak topology.
Indeed, for any such monomial, we have $$t_{i_{1}, j_{1}} t_{i_{2}, j_{2}} \cdots t_{i_{k}, j_{k}} = (Te_{j_{1},i_{2}} Te_{j_{2},i_{3}} \cdots T e_{j_{k-1}i_{k}}T)_{i_{1}, j_{k}}$$
so the convergence in $\mathcal{B}$-valued distribution to $T$ implies convergence of these non-commutative $\ast$ polynomials.

This provides us with an obvious candidate for our affiliated operator, namely $$a := t_{1,2}^{-1} - i .$$  Thus, we need to show the following:
\begin{enumerate}
\item\label{prove1}  $T^{-1} \in M_{2}( \tilde{ \mathcal{A}} )^{sa}$.
\item\label{prove2} $ t_{1,2}^{-1} - i  \in \tilde{ \mathcal{A}}^{sa}$.
\item\label{prove3} $G^{(n)}_{ t_{1,2}^{-1} - i}(b) = G^{(n)}(b)$ for all $b \in \Omega_{n}(\mathcal{B})$ (resp. $H^{+}(\mathcal{B})$).
\end{enumerate}

By Lemma \eqref{affiliation}, \eqref{prove1} will follow if we show that $\ker{(T)} = \{ 0 \}$.
Appealing to the scalar valued case, consider the Cauchy transform
$$g_{T}(z) := \tau_{2} \circ G_{T}(z1_{2}) = \int_{\mathbb{R}} \frac{1}{z-t} d\nu_{T}(t): \mathbb{C}^{+} \mapsto \mathbb{C}^{-}$$
where $e_{T}$ is the functional calculus of $T$,  and $\nu_{T}$ is the measure defined by the distribution $$ \nu_{T}((-\infty, s ) ) = \tau_{2} \circ e_{T}((-\infty, s)). $$
A  non-empty kernel is equivalent to the condition $$\tau_{2} \circ e_{T}(\{0 \}) = \delta > 0.$$
Thus, for $\epsilon > 0$ , we have that
$$ -\Im{ ( g_{T}(i\epsilon 1_{2}) )} = \int_{\mathbb{R}} \frac{\epsilon}{t^{2} + \epsilon^{2}} d\nu_{T}(t) \geq \frac{ \nu_{T}(\{ 0\})}{\epsilon} =  \frac{ \tau_{2}(e_{T}(\{ 0 \} ))}{\epsilon} = \frac{\delta}{\epsilon} $$

Utilizing similar constructions for the operators $A_{k}$,
let $e_{a_{k}}$ denote the functional calculus of the operator $a_{k}$ and $\nu_{k}$ the corresponding distribution.
First, observe that 
\begin{align} \nonumber (A_{k} - z1_{2})^{-1} &=  \left[ \begin{array}  {cc}
   z & -(a_{k} + i)^{-1} \\
  - (a_{k} -  i)^{-1} & z
   \end{array} \right]^{-1} \\ &=   \left[ \begin{array}  {cc}
   z(z^{2} - (a_{k}^{2} + 1)^{-1} )^{-1} & (a_{k} + i)^{-1}(z^{2} - (a_{k}^{2} + 1)^{-1} )^{-1} \\
  (a_{k} + i)^{-1}(z^{2} - (a_{k}^{2} + 1)^{-1} )^{-1}  & z(z^{2} - (a_{k}^{2} + 1)^{-1} )^{-1}
   \end{array} \right].   
\end{align}
Thus, we have that
$$g_{A_{k}}(z) := \tau_{2} \circ G_{A_{k}}(z1_{2}) = \tau_{2}(  (A_{k} - z1_{2})^{-1}) = \int_{\mathbb{R}} \frac{z}{z^{2} - \frac{1}{t^{2} + 1}}d\nu_{k}(t).$$
Therefore, 
\begin{align}
-\Im{(g_{A_{k}}(i\epsilon1_{2}))} &=   \int_{\mathbb{R}} -\Im{\left( \frac{(i\epsilon)\left[(-i\epsilon)^{2} - \frac{1}{t^{2} + 1}\right]}{ \left(-\epsilon^{2} - \frac{1}{t^{2} + 1}\right)^{2}}d\nu_{k}(t) \right)}\\
&=   \int_{\mathbb{R}} \left( \frac{\epsilon}{ \epsilon^{2} + \frac{1}{t^{2} + 1}}\right)d\nu_{k}(t) \\
&=\label{babe} \int_{\mathbb{R}} \frac{\epsilon(t^{2} + 1)}{\epsilon^{2}(t^{2} + 1) + 1}d\nu_{k}(t).
\end{align}

Now, condition \eqref{tightness} is equivalent to tightness of the family of measures $\{ \nu_{k} \}_{k \in \mathbb{N}}$
so that we may assume that this family subconverges to a measure $\nu$.  Since $G_{A_{k}}$ converges to $G_{T}$ in the pointwise weak topology, we  conclude that $$g_{\nu}(z) := \int_{\mathbb{R}} \frac{z}{z^{2} - \frac{1}{t^{2} + 1}}d\nu(t) = g_{T}(z).$$  Thus, for all $\epsilon > 0$, we have that \begin{equation}\label{oldstyle}
 -\Im(g_{\nu}(i\epsilon)) = -\Im( g_{T}(i\epsilon)) \geq \frac{\delta}{\epsilon} .\end{equation}  
However, pick $N > 0$ satisfying the following:
$$\nu([-N,N]) > 1 - \delta/2 .$$
Combining \eqref{babe} and \eqref{oldstyle}, we have that
\begin{align}
0 & \leq \int_{\mathbb{R}} \frac{\epsilon(t^{2} + 1)}{\epsilon^{2}(t^{2} + 1) + 1}d\nu(t)  - \frac{\delta}{\epsilon} \\ &\leq \frac{\nu\left(\mathbb{R}\setminus [-N,N] \right)}{\epsilon} + \frac{\epsilon(N^{2} + 1)\nu([-N,N])}{\epsilon^{2}(N^{2} + 1) + 1} - \frac{\delta}{\epsilon}\\
&\leq \frac{\epsilon(N^{2} + 1)}{\epsilon^{2}(N^{2} + 1) + 1} - \frac{\delta}{2\epsilon}
\end{align}
This implies
\begin{equation}
\frac{\delta}{2} \leq \frac{\epsilon^{2}(N^{2} + 1)}{\epsilon^{2}(N^{2} + 1) + 1}
\end{equation}
However, the right hand side converges to $0$ as $\epsilon \downarrow 0$.
This contradiction implies that $T$ has no kernel. 
  This completes the proof of \eqref{prove1}.

Now, note that $T^{-1} \in M_{2}( \mathcal{ \tilde{A}} )^{sa}$, which implies
$$ (t_{1,2}^{-1} - i)e_{1,1} = e_{1,2} \left( T^{-1} + \left[ \begin{array}  {cc}
    0 & i \\
  -i & 0 
   \end{array} \right] \right) e_{2,1} \in  M_{2}( \mathcal{ \tilde{A}} )^{sa} .$$
Property \eqref{prove2} is an immediate consequence.


Regarding claim \eqref{prove3}, we will prove this under the assumption that \eqref{convergence} holds on $\Omega(\mathcal{B})$.  
Define
$$ b = \left[ \begin{array}  {cc}
   \lambda i & 0 \\
 0 & -\gamma i
   \end{array} \right]  \otimes1_{n} + \delta \in \Omega_{n}(\mathcal{B}). $$
where $\lambda , \gamma \in \mathbb{R}^{+}$ and $\|\delta \| << \max\{ \lambda , \gamma\}$.  Further define
$$  B = \left[ \begin{array}  {cc}
    0_{2n} & b^{\ast} \\
  b & 0_{2n} 
   \end{array} \right] \ ; \ \ B' = B - \left[ \begin{array}  {cc}
    0_{2n} & i1_{2n} \\
  -i1_{2n} & 0_{2n} 
   \end{array} \right] =  \left[ \begin{array}  {cc}
    0_{2n} & b^{\ast} -  i1_{2n} \\
  b + i1_{2n} & 0_{2n}
   \end{array} \right]$$.

Observe that \eqref{jointmom} implies that the 
$$ \left[ \begin{array}  {cc}
    0_{2n} & ( a_{k} + i)^{-1}  \otimes1_{2n}\\
(a_{k} - i)^{-1}  \otimes1_{2n}& 0_{2n}
   \end{array} \right]  \rightarrow \left[ \begin{array}  {cc}
    0_{2n} & t_{1,2} \otimes1_{2n}\\
  t_{2,1} \otimes1_{2n}& 0_{2n}
   \end{array} \right] $$
in the pointwise weak topology.

We have that
\begin{align} 
 E_{4n} \label{LINE1} & \left[ \left(\left[ \begin{array}  {cc}
    0_{2n} & (t_{1,2}^{-1} + i) \otimes1_{2n}\\
  (t_{2,1}^{-1} -i ) \otimes1_{2n}& 0_{2n} 
   \end{array} \right] - B \right)^{-1}\right] \\ 
&= \label{LINE2} E_{4n}\left[\left( \left[ \begin{array}  {cc}
    0_{2n} & t_{1,2}^{-1} \otimes1_{2n}\\
  t_{2,1}^{-1} \otimes1_{2n}& 0_{2n}
   \end{array} \right] - B'\right)^{-1}\right] \\
&=  \label{LINE3}    -B'^{-1} +  B'^{-1} E_{4n}\left[\left( B'^{-1} - \left[ \begin{array}  {cc}
    0_{2n} & t_{1,2} \otimes1_{2n}\\
  t_{2,1} \otimes1_{2n}& 0_{2n}
   \end{array} \right] \right)^{-1}\right] B'^{-1} \\
&=    \label{LINE4}  \lim_{k\uparrow\infty} -B'^{-1} +  B'^{-1} E_{4n}\left[\left( B'^{-1} - \left[ \begin{array} {cc}
    0_{2n} & (a_{k} + i)^{-1} \otimes1_{2n}\\
  (a_{k} - i)^{-1} \otimes1_{2n}& 0_{2n}
   \end{array} \right] \right)^{-1}\right]B'^{-1} \\
 &=   \label{LINE5}  \lim_{k\uparrow\infty}  E_{4n}\left[\left( \left[ \begin{array}  {cc}
    0_{2n} & (a_{k} + i) \otimes1_{2n}\\
 (a_{k} - i) \otimes1_{2n}& 0_{2n}
   \end{array} \right] - B'\right)^{-1}\right] \\
 &=  \label{LINE6} \lim_{k\uparrow\infty} E_{4n}\left[\left( \left[ \begin{array}  {cc}
    0_{2n} & a_{k}  \otimes1_{2n}\\
a_{k}  \otimes1_{2n}& 0_{2n}
   \end{array} \right] - B\right)^{-1} \right] \\
&=  \label{LINE7} \left[ \begin{array}  {cc}
    0_{n} & G^{(2n)}(b^{\ast})\\
 G^{(2n)}(b)& 0_{n}
   \end{array} \right] 
\end{align}
Each line is consistent with the domains of the various functions.  Indeed, for \eqref{LINE1} and \eqref{LINE2},  $ (t_{1,2}^{-1} + i) \otimes1_{2n}$ is a self adjoint operator and $b \in \Omega_{n}(\mathcal{B})$ is always in the resolvent set, so this extends to the matrix inverse (the same argument works for \eqref{LINE5} and \eqref{LINE6}).
For \eqref{LINE3} , note that, for $\delta = 0$,
$$\left(B'^{-1} - \left[ \begin{array}  {cc}
    0_{2n} & t_{1,2} \otimes1_{2n}\\
  t_{2,1} \otimes1_{2n}& 0_{2n}
   \end{array} \right] \right)^{-1}$$ is invertible for $\lambda , \gamma > 1$ with inverse equal to $$
\left[ \begin{array}  {cc} 0_{2n} & \left[ \begin{array}  {cc}
  \frac{ i}{1 + \lambda} - t_{1,2}& 0 \\ 
 0 & \frac{i}{1-\gamma} -t_{1,2}
   \end{array} \right]^{-1}  \otimes1_{n} \\
  \left[ \begin{array}  {cc}
  \frac{- i}{1 + \lambda} - t^{\ast}_{1,2}& 0 \\ 
 0 & \frac{-i}{1-\gamma} -t^{\ast}_{1,2}
   \end{array} \right]^{-1}  \otimes1_{n}  &  0_{2n}
\end{array} \right].$$
Thus, for $\delta$ small, this is still invertible.  Again, the same argument works for \eqref{LINE4}.
The convergence of the Cauchy transforms follows from convergence in the pointwise weak topology by Proposition \eqref{GConvExt}.

Now, equating \eqref{LINE1} and \eqref{LINE7}, specifically the lower left blocks,  we conclude that \eqref{prove3} holds for the case $\Omega(\mathcal{B})$ by analytic continuation.  

The case where assumption \eqref{convergence} from the statement of the theorem holds only on $H^{+}(\mathcal{B})$  follows in a similar manner.  Indeed, \eqref{LINE1} - \eqref{LINE7}, should have $n$ replace with $n/2$.  Each line makes sense for $b \in M_{n}^{+}(\mathcal{B})$, we have that $b + i1_{n} \in M_{n}^{+}(\mathcal{B})$ so that \eqref{LINE3} and \eqref{LINE4} so that these equalities are consistent with the domains.   The invertibility of the integrands in \eqref{LINE1}, \eqref{LINE2}, \eqref{LINE5} and \eqref{LINE6} 
follows from the invertibility of the fact that the $1,2$ and $2,1$ blocks are invertible and the $1,1$ and $2,2$ blocks are $0$.

This completes the proof of our theorem.

\end{proof}

\begin{remark}
Note that assumption \eqref{tightness} may not be weakened.  Indeed, if we consider the atomic measures $ \mu_{k} = \delta_{0}/2 + \delta_{k}/2$, 
we have that $$ G_{\mu_{k}}(z) = \frac{1}{2z} + \frac{1}{2(z-k)} \mapsto \frac{1}{2z}$$
where the convergence is as $k \uparrow \infty$, uniformly on sets of the form $\mathcal{O}$.  The limit point is not a Cauchy transform as it does not have the appropriate asymptotics.

Assumption \eqref{tightness} will be of importance later in this work and will be referred to as \textit{tightness}, with various modifiers.

Also note that, \eqref{convergence} is only assumed on $H^{+}(\mathcal{B})$, we can recover an unbounded operator whose Cauchy transform agrees on $H^{+}(\mathcal{B})$ but, as we shall see in the next section, this Cauchy transform may have a distinct extension to the full resolvent set.
\end{remark}

\section{Counterexamples Arising from Cauchy Distributions}\label{CounterE}

Let $X$ denote a random variable with the Cauchy distribution.
Note that the Cauchy transform of $X$ satisfies
$$ \phi((z-X)^{-1}) = \frac{1}{z \pm i} $$
for $z \in \mathbb{C}^{\pm}$.

We will utilize Boolean independence in what follows (see \cite{SBool} for the relevant definitions).

\begin{proposition}\label{FBCS}
Let $X_{1} , X_{2} , \ldots , X_{n} \in \tilde{\mathcal{A}}^{sa}$ denote standard Cauchy distributed random variables in a $C^{\ast}$-probability space $(\mathcal{A}, \phi)$, where $\mathcal{A}$ is assumed to be large enough to maintain free, Boolean or classical independence for these variables.
Then, for any collection $z_{1}, z_{2}, \ldots, z_{k} \in \mathbb{C}^{+}$, the moment
$$ \phi\left( (z_{1} - X_{i_{1}})^{-1} (z_{2} - X_{i_{2}})^{-1} \cdots (z_{k} - X_{i_{k}})^{-1} \right) $$
is equal to the  fixed value 
$$\prod_{j=1}^{k}  (z_{j} + i )^{-1}$$  for  random variables satisfying each of the following:
\begin{enumerate}
\item $X_{1} = X_{2} = \cdots = X_{n}$
\item $X_{1} , X_{2} , \ldots , X_{n}$ are classically independent.
\item $X_{1} , X_{2} , \ldots , X_{n}$ are freely independent.
\item $X_{1} , X_{2} , \ldots , X_{n}$ are Boolean independent.
\end{enumerate}
\end{proposition}

\begin{proof}
We begin by showing that, if $X_{1} = X_{2} = \cdots = X_{n}$, then
\begin{equation}\label{onevarclaim}
\phi\left( (z_{1} - X_{i_{1}})^{-1} (z_{2} - X_{i_{2}})^{-1} \cdots (z_{k} - X_{i_{k}})^{-1} \right) = \prod_{j=1}^{k}  (z_{j} + i )^{-1}.
\end{equation}
We may assume that the $z_{j}$ are all distinct and the full result will follow through continuity.  
With this assumption, using partial fractions, consider $\lambda_{1} , \lambda_{2} , \ldots , \lambda_{k} \in \mathbb{C}$ such that
$$(z_{1} - X_{i_{1}})^{-1} (z_{2} - X_{i_{2}})^{-1} \cdots (z_{k} - X_{i_{k}})^{-1} = \frac{\lambda_{1}}{z_{1} - X_{i_{1}}} + \frac{\lambda_{2}}{z_{2} - X_{i_{2}}} + \cdots + \frac{\lambda_{k}}{z_{k} - X_{i_{k}}} .$$
Note that
\begin{align*} \phi\left( \frac{\lambda_{1}}{z_{1} - X_{i_{1}}}  + \frac{\lambda_{2}}{z_{2} - X_{i_{2}}} + \cdots + \frac{\lambda_{k}}{z_{k} - X_{i_{k}}} \right) = \frac{\lambda_{1}}{z_{1} + i} + \frac{\lambda_{2}}{z_{2} + i} + \cdots & + \frac{\lambda_{k}}{z_{k} + i} .\end{align*}
Reversing the partial fraction decomposition, claim \eqref{onevarclaim} follows.

We next assume that $X_{1} , X_{2} , \ldots, X_{n}$ are classically independent.  
Rearrange the product
$$(z_{1} - X_{i_{1}})^{-1} (z_{2} - X_{i_{2}})^{-1} \cdots (z_{k} - X_{i_{k}})^{-1}$$
into $n$ blocks $P_{j}$, arranged by the index of $X_{i_{p}}$.
Then
$$ \phi\left((z_{1} - X_{i_{1}})^{-1} (z_{2} - X_{i_{2}})^{-1} \cdots (z_{k} - X_{i_{k}})^{-1} \right)= \prod_{j=1}^{n} \phi(P_{j})$$
by classical independence.  The result then reduces to \eqref{onevarclaim}.

To address the  Boolean case, we rewrite the product 
$$(z_{1} - X_{i_{1}})^{-1} (z_{2} - X_{i_{2}})^{-1} \cdots (z_{k} - X_{i_{k}})^{-1} $$
as follows.  Let $j$ denote the largest number so that $i_{1} = i_{2} = \ldots = i_{j}$.  Let $$ M_{1} =  (z_{1} - X_{i_{1}})^{-1} (z_{2} - X_{i_{2}})^{-1} \cdots (z_{k} - X_{i_{j}})^{-1}.$$
Repeating this process, we rewrite this as an alternating product of blocks with the same index,
$$(z_{1} - X_{i_{1}})^{-1} (z_{2} - X_{i_{2}})^{-1} \cdots (z_{k} - X_{i_{k}})^{-1}  = M_{1}M_{2} \cdots M_{\ell}$$
with $\ell \leq k$.
We have that 
\begin{align*}
\phi\left[ (z_{1} - X_{i_{1}})^{-1} (z_{2} - X_{i_{2}})^{-1} \cdots (z_{k} - X_{i_{k}})^{-1} \right]  & = \phi( M_{1}M_{2} \cdots M_{\ell}) \\ &= \phi( M_{1}) \phi(M_{2}) \cdots \phi(M_{\ell})
\end{align*}
by Boolean independence, and our claim again reduces to \eqref{onevarclaim}.

Thus, we are left with the free case.  We once again decompose our product
$$(z_{1} - X_{i_{1}})^{-1} (z_{2} - X_{i_{2}})^{-1} \cdots (z_{k} - X_{i_{k}})^{-1}  = Q_{1} Q_{2} \cdots Q_{p} $$
where the terms in each of the $Q_{j}$ have the same index but where it is \textit{no longer assumed} that the indices alternate.
Let $$\overline{Q_{j}} = Q_{j} - \phi(Q_{j}). $$
We claim that
\begin{equation}\label{multvarclaim} \phi(\overline{Q_{1}} \ \overline{Q_{2}} \cdots \overline{Q_{p}}) = 0 \end{equation} 
even without the assumption that the indices alternate.

Proceeding by induction, the $p=1$ case is immediate.
For $p=2$, if the indices differ, then this follows from freeness.  
If the indices are the same, polarize the term so that
$$  \phi \left(\overline{Q_{1}} \ \overline{Q_{2}}\right) = \phi \left(Q_{1}\overline{Q_{2}} \right) - \phi \left(Q_{1}\right)\phi \left(\overline{Q_{2}}\right) = \phi \left(Q_{1}\overline{Q_{2}}\right) = \phi \left(Q_{1}Q_{2}\right) - \phi \left(Q_{1} \right)\phi \left(Q_{2} \right).$$
It is immediate from \eqref{onevarclaim} that the right hand side is $0$, proving \eqref{multvarclaim} for $p=2$.

For general $p$, if the indices are indeed alternating, then the product is $0$ by freeness.  Otherwise, we may assume that the indices associated to $Q_{1}$ and $Q_{2}$ are the same (since free random variables can always be realized in a tracial setting).
\begin{align}
 \phi(\overline{Q_{1}} \ \overline{Q_{2}} \cdots \overline{Q_{p}}) &= \phi(Q_{1} \ \overline{Q_{2}} \cdots \overline{Q_{p}}) - \phi(Q_{1})\phi(\overline{Q_{2}} \ \overline{Q_{3}} \cdots \overline{Q_{p}}) \\
&= \phi(Q_{1} \ \overline{Q_{2}} \cdots \overline{Q_{p}}) \label{inductionhyp} \\
&= \phi(Q_{1} Q_{2} \overline{Q_{3}} \cdots \overline{Q_{p}}) - \phi(Q_{2})\phi(Q_{1} \overline{Q_{3}} \cdots \overline{Q_{p}}) \label{laststep}
\end{align}
where \eqref{inductionhyp} is by induction.
Let $R_{1}$ take on either $Q_{1}Q_{2}$ or $Q_{1}$.  Since $Q_{1}$ and $Q_{2}$ have the same index, we have that
$$ 0 = \phi(\overline{R_{1}} \  \overline{Q_{3}} \cdots \overline{Q_{p}}) =  \phi(R_{1} \  \overline{Q_{3}} \cdots \overline{Q_{p}}) -  
\phi(R_{1}) \phi( \overline{Q_{3}} \cdots \overline{Q_{p}}). $$
Applying this equality to both terms in \eqref{laststep}, we have that  \eqref{laststep} is equal to 
$$ \phi(Q_{1} Q_{2}) \phi( \overline{Q_{3}} \cdots \overline{Q_{p}}) - \phi(Q_{1}) \phi(Q_{2}) \phi( \overline{Q_{3}} \cdots \overline{Q_{p}}) = 0 $$
by induction.
This proves \eqref{multvarclaim}.

Returning to the main claim for the free case, we once again take an alternating decomposition 
$$(z_{1} - X_{i_{1}})^{-1} (z_{2} - X_{i_{2}})^{-1} \cdots (z_{k} - X_{i_{k}})^{-1}  = M_{1}M_{2} \cdots M_{\ell}$$
with $\ell \leq k$.  Letting $$ M_{j}^{(0)} = \overline{M_{j}} \ ; \ \   M_{j}^{(1)} = \phi(M_{j}),$$
we have that 
$$\phi( M_{1}M_{2} \cdots M_{\ell}) = \sum_{ \epsilon_{1} , \epsilon_{2} , \ldots , \epsilon_{\ell} = 1}^{2} \phi \left( M_{1}^{\epsilon_{1}} M_{2}^{\epsilon_{2}} \cdots M_{\ell}^{\epsilon_{\ell}} \right) . $$
For a fixed term $$ \phi \left( M_{1}^{\epsilon_{1}} M_{2}^{\epsilon_{2}} \cdots M_{\ell}^{\epsilon_{\ell}} \right),$$
we may factor out the terms with $\epsilon_{j} = 1$ (since these are just scalars).  The remaining terms (if there are any) all satisfy $\epsilon_{p} = 0$, so that this  is a special case of \eqref{multvarclaim}.  Thus, if any of the $\epsilon_{p} = 0$, the term disappears so that the only non-zero contribution is when $\epsilon_{1} = \epsilon_{2} = \cdots = \epsilon_{\ell} = 1$.
Thus,
$$\phi( M_{1}M_{2} \cdots M_{\ell}) = \phi( M_{1}) \phi( M_{2}) \cdots \phi(M_{\ell}) $$
and our theorem then follows from \eqref{onevarclaim}.
\end{proof}

\begin{corollary}
Let $X_{1} , X_{2} , \ldots , X_{n} \in \tilde{\mathcal{A}}^{sa}$  denote standard Cauchy distributed random variables in a $C^{\ast}$-probability space $(\mathcal{A}, \phi)$where $\mathcal{A}$ is assumed to be large enough to maintain free, Boolean or classical independence for these variables..
Consider the operator-valued probability space $(M_{n}(\mathcal{A}), E, M_{n}(\mathbb{C}))$
where $E = \phi \otimes 1_{n}$ and $\mathcal{B} = M_{n}(\mathbb{C})$.
Let \begin{equation} X = \left[ \begin{array}  {ccccc}
    X_{1} & 0 & 0 & \cdots & 0  \\
 0  & X_{2} & 0 & \cdots & 0 \\
\ & \vdots & \ & \vdots & \ \\
0  & 0 & 0 & \cdots & X_{n} \\
   \end{array} \right]   \in M_{n}(\mathcal{A})^{sa}.\end{equation}
Then, for all $B\in M_{nk}^{+}(\mathcal{B})$, we have that
$$ G_{X}^{(k)}(B) =  (B + i1_{nk})^{-1}$$
for $X_{i}$ satisfying each of the following:
\begin{enumerate}
\item $X_{1} = X_{2} = \cdots = X_{n}$
\item $X_{1} , X_{2} , \ldots , X_{n}$ are classically independent.
\item $X_{1} , X_{2} , \ldots , X_{n}$ are freely independent.
\item $X_{1} , X_{2} , \ldots , X_{n}$ are Boolean independent.
\end{enumerate}
In particular, the restriction $$ G_{X}|_{H^{+}(M_{n}(\mathbb{C}))} $$
does not distinguish the associated operator algebras $\{\{X\} \cup M_{n}(\mathbb{C}) \}''.$
\end{corollary}
\begin{proof}
Let $B = (b_{i,j})_{i,j = 1}^{nk} \in M_{nk}^{+}(\mathbb{C})$.
Let this algebra act on $\mathbb{C}^{nk}$ with orthonormal basis $\{ v_{i} \}_{i=1}^{nk}$.
Note that the associated vector states satisfy
$$ 0 < \Im{(\langle Bv_{i} , v_{i} \rangle)} = \Im{(b_{i,i})}.$$

We rewrite $B = D + B'$ where $D$ is diagonal and $B' \in M_{nk}(\mathbb{C})$ has $0$ on the diagonal entries.
We have
\begin{align}
 E[(B - X\otimes1_{k})^{-1}] & = E[(D - X\otimes1_{k}) + B')^{-1}]  \\ &= E[(D - X\otimes1_{k})^{-1} (1 - B'(D - X\otimes1_{k})^{-1} )^{-1}]
\end{align}
If we take an open neighborhood where $$\inf_{i=1, \ldots , nk} |b_{i,i}| >> \sup_{\ell \neq p} |b_{\ell,p}|$$
(a small open neighborhood of a diagonal $B$ , for instance), then 
the term $$(1 - B'(D - X\otimes1_{k})^{-1} )^{-1} $$
may be written as a geometric series.
Note that $$ (D - X\otimes1_{k})^{-1} =  \sum_{i = 1}^{n} \sum_{j=1}^{k}  (b_{i+j , i+j} - X_{i})^{-1}\otimes e_{i + j , i + j}$$ 
where $\{ e_{\ell , p} \}_{\ell , p = 1}^{nk}$ are the matrix units for $M_{nk}(\mathbb{C})$.
This implies that $$[B'(D - X\otimes1_{k})^{-1}]^{p} $$ 
is a complex polynomial in the variables $$ (b_{i + 1, i + 1} - X_{1})^{-1} ,  (b_{i + 2, i + 2} - X_{2})^{-1} , \ldots ,  (b_{i + n, i + n} - X_{n})^{-1} $$ for $i = 1, 2, \ldots , k$.  Moreover, each of the diagonal entries satisfy $\Im{(b_{j,j})} > 0$ for $j=1,2, \ldots , nk$.
By \eqref{FBCS}  , we have that
$$E ( [B'(D - X\otimes1_{k})^{-1}]^{p} ) = [B'(D + i1_{nk})^{-1}]^{p}. $$
We conclude that 
$$ G_{X}^{(k)}(B) =   (B + i1_{nk})^{-1}$$
on this open neighborhood and the full claim follows by analytic continuation.

\end{proof}

\begin{remark}
We are unable at this time to compute the convolution of operators $$ X_{1} = \left[ \begin{array}  {cc}
    x_{1} & 0 \\
 0 & 0 \\
   \end{array} \right] \  ; \ \    X_{2} = \left[ \begin{array}  {cc}
    0 & 0 \\
 0 & x_{2} \\
   \end{array} \right]  $$
where $x_{1}$ and $x_{2}$ are free copies of the Cauchy distribution in a space $(\mathcal{A} , \tau)$ [the conditional expectation in this setting is $\tau \otimes 1_{2}$.  This is generally a somewhat difficult question since the Cauchy transform must generally be known for all matrix dimensions $n$.  The only positive results of this type may be found in \cite{AWJac} where this problem was circumvented via combinatorial methods.
\end{remark}

\section{The R-transform for Affiliated Operators with B Finite Dimensional.}\label{RTransform1}

\begin{theorem}\label{Rexist}
Assume that $(\mathcal{A}, E, \mathcal{B})$ is an operator valued, tracial  probability space where $\mathcal{B}$ is finite dimensional.
For $\lambda \in \mathbb{R}^{+}$, define
$$d_{n}(\lambda) = i\lambda \left[ \begin{array}  {cc}
    1 & 0 \\
 0 & -1 \\
   \end{array} \right]\otimes1_{n} \in \Omega_{n}(\mathcal{B}). $$
Let $\{ X_{i}\}_{i \in I} \subset  \tilde{\mathcal{A}}^{sa}$ denote a family of random variables where the scalar distributions $\{ \tau \circ e_{X_{i}} \}_{i \in I}$
are a tight family.  Then, there exists a sequence $\lambda_{n} \in \mathbb{R}^{+}$, decreasing over $n$,
and a positive decreasing function $$p_{n}: (0, \lambda_{n})  \mapsto \mathbb{R}^{+}$$
so that $p_{n}(\lambda) = O(\lambda)$  and the set
 $$B_{p_{n}(\lambda)}(d_{n}(\lambda))$$ is in the domain of $\mathcal{R}_{X_{i}}$ for all $n \in \mathbb{N}$,
 $\lambda \leq \lambda_{n}$ and  $i \in I$.
\end{theorem}
\begin{proof}

We define a function
$K = \{ K^{(n)} \}_{n=1}^{\infty}$  with domain $\Omega(\mathcal{B})^{-1}$  as follows:
$$ K^{(n)}(w):=  G^{(n)}(w^{-1}) .$$

Our main tool for controlling our inverse will be Theorem \eqref{Bloch}.  To do so, we first  show that there exists
a $\lambda_{n} > 0$ such that
\begin{equation}\label{IFTclaim}
\| \delta K^{(2n)}(d_{n}(\lambda) ;\cdot ) - Id \| < \epsilon
\end{equation}
for all $\lambda \leq \lambda_{n}$.  Moreover, the choice of $\lambda_{n}$ is uniform
for random variables with tight $\tau \circ e_{X}$ .


Observe that, for $h \in M_{2n}(\mathcal{B}) $ with  $\|h\| \leq 1$, we have
$$ \delta K^{(2n)}(w;h)  = E_{n}\left[ (w^{-1} -X_{n})^{-1} w^{-1}hw^{-1}(w^{-1}-X_{n})^{-1} \right]$$
Thus, we have that
\begin{align}
\| \delta K^{(2n)}(d_{n}& (\lambda);h)  - h\|  = \| E_{n}[(1 - d_{n}(\lambda) X_{n})^{-1}h(1 - X_{n}d_{n}(\lambda))^{-1} - h] \| \\
&\leq \label{exampleterm} \| E_{n}[ (d_{n}(\lambda) X_{n})(1 - d_{n}(\lambda) X_{n})^{-1}h(1 -X_{n}d_{n}(\lambda))^{-1}] \| \\ 
& \label{exampletermb} +  \| E_{n}[(1 - d_{n}(\lambda) X_{n})^{-1}h(1 - X_{n}d_{n}(\lambda))^{-1} (X_{n}d_{n}(\lambda))] \|  \\
& \label{exampletermc}  +  \| E_{n}[  (d_{n}(\lambda) X_{n})(1 - d_{n}(\lambda) X_{n})^{-1}h(1 - X_{n}d_{n}(\lambda))^{-1} (X_{n}d_{n}(\lambda))] \| 
\end{align}
Focusing on a \eqref{exampleterm}, let $ \{p_{i, 2n}\}_{i=1}^{N(2n) } \subset M_{2n}(\mathcal{B})$ denote positive vectors that  induce the weak topology as vector states.  That is,  $\|p_{i,2n} \|_{L^{2}} = 1$ and let $$C_{n} = \max_{i=1, \ldots , N(2n)} \|\ p_{i,2n} \|$$
where this is taken to be the norm on $\mathcal{A}$.
\begin{align}
\| E_{n}[ (d_{n}(\lambda) X_{n})&(1 - d_{n}(\lambda) X_{n})^{-1}h(1 -X_{n}d_{n}(\lambda))^{-1}] \| \\ & = \sup_{i=1}^{N(2n)} \tau\left(  E_{n}[ (d_{n}(\lambda) X_{n})(1 - d_{n}(\lambda) X_{n})^{-1}h(1 -X_{n}d_{n}(\lambda))^{-1}] p_{i} \right) \nonumber \\
&= \sup_{i=1}^{N(2n)} \nonumber \tau\left(  (d_{n}(\lambda) X_{n})(1 - d_{n}(\lambda) X_{n})^{-1}h(1 -X_{n}d_{n}(\lambda))^{-1} p_{i} \right) \\
&\leq \label{obiwan} \|h\| C_{n} \tau\left( \left[\frac{d_{n}(\lambda) X_{n}}{1 -d_{n}(\lambda) X_{n}}\right]^{2} \right)^{1/2} \tau\left( \left[\frac{1}{1 - d_{n}(\lambda) X_{n}}\right]^{2} \right)^{1/2}
\end{align}
where \eqref{obiwan} follows through Cauchy Schwarz, traciality and positivity of $\tau$.
Also note that $d_{n}(\lambda)$ and $X_{n}$ are diagonal and $d_{n}(\lambda)$ has scalar entries, so we are in the commutative setting.
Utilizing the functional calculus $e_{X}$ and focusing on a single one of these diagonal entries,  we have
\begin{equation}\label{obiwanA}
\left| \tau\left( \left[\frac{1}{1 - i\lambda X}\right] \right) \right| \leq  \int_{\mathbb{R}} \left| \frac{1}{1 - i\lambda t} \right| \tau \circ e_{X}(t) \leq 1
\end{equation}
\begin{equation}\label{obiwanB}
\lim_{\lambda \downarrow 0} \tau\left( \left[\frac{i\lambda X}{1 - i\lambda X}\right] \right) = \lim_{\lambda \downarrow 0} \int_{\mathbb{R}} \frac{i\lambda t}{1 - i\lambda t}\tau \circ e_{X}(t) = 0
\end{equation}
where the convergence in \eqref{obiwanB} is uniform for tight $\tau \circ e_{X}$.
Thus, \eqref{obiwan} is smaller than $\epsilon/3$ for $\lambda < \lambda_{n}$ small enough
so that the same holds for  \eqref{exampleterm}.  Similar proofs bound \eqref{exampletermb} and \eqref{exampletermc} so that
\eqref{IFTclaim} holds.

In order to invoke  Theorem \eqref{Bloch}, we must show that $K$ is bounded.  Define
\begin{equation}
\tilde{K}^{(2n)}(y) = K^{(2n)}(d_{n}(\lambda) + y)  - K^{(2n)}(d_{n}(\lambda)): B_{R(\lambda)}(\{ 0 \}) \mapsto B_{M(\lambda)}(\{0\}) 
\end{equation}

Let $R(\lambda) = \lambda/2$.
Determining $M(\lambda)$, observe that
\begin{align}
\|\tilde{K}^{(n)}& (y)  \| =  \| G^{(n)}(d_{n}(\lambda)+y)^{-1} - G^{(n)}(d_{n}(\lambda))^{-1}) \|  \label{MF1} \\
&= \label{MF2} \left\|E_{n} \left[  [(d_{n}(\lambda)+y)^{-1} - X]^{-1} (d_{n}(\lambda) + y)^{-1} yd_{n}(\lambda)^{-1}  [d_{n}(\lambda)^{-1} - X]^{-1}\right]  \right\| \\
& \label{MF3} \leq \frac{\lambda}{2} \| [(d_{n}(\lambda) + y)^{-1} - X]^{-1} (d_{n}(\lambda) + y)^{-1}\| \| d_{n}(\lambda)^{-1}  [d_{n}(\lambda)^{-1} - X]^{-1}\| 
\end{align}
where \eqref{MF3} is contractivity of $E$.

Observe that $$\sigma(d_{n}(\lambda) + y) \subset B_{\lambda/2}(\pm i\lambda) $$
so that
$$\sigma((d_{n}(\lambda) + y)^{-1}) =\left\{ \frac{1}{z} : z\in \sigma((d_{n}(\lambda) + y))\right\}  \subset  \left\{ \frac{1}{z} : z\in  B_{\lambda/2}(\pm i\lambda)\right\}$$
by the spectral mapping theorem (3.3.6 in \cite{KRe}, for instance).
Since the norm agrees with the spectral radius in a C$^{\ast}$-algebra, we have that
\begin{equation}\label{jake}
\| (d_{n}(\lambda) + y)^{-1} \| \leq \sup_{z \in B_{\lambda/2}(\pm i\lambda)} \left| \frac{1}{z}\right| = \frac{2}{\lambda}.
\end{equation}
Moreover, using the spectral mapping theorem,
$$\Im((d_{n}(\lambda) + y)^{-1} - X)) = \Im((d_{n}(\lambda) + y)^{-1}) \geq \inf_{z \in B_{\lambda/2}(\pm i\lambda)} \Im{\left( \frac{1}{z} \right)}= \frac{2}{3\lambda} 1_{2n}$$
and, using Lemma \eqref{affresolvent}, we conclude
\begin{equation}\label{elwood}
 \| [(d_{n}(\lambda) + y)^{-1} - X]^{-1} \| \leq \frac{3\lambda}{2}
\end{equation}
A similar proof shows that 
$$ \| d_{n}(\lambda)^{-1}  [d_{n}(\lambda)^{-1} - X]^{-1}\|  \leq 1 $$
Combining \eqref{jake} and \eqref{elwood}, we have that
\eqref{MF3} is bounded by 
$$ M(\lambda)  = \frac{3\lambda}{2}. $$

By Theorem \eqref{Bloch},
we have that \begin{equation}\label{setcontainment}
 B_{P(\lambda)}(K^{(n)}(d_{n}(\lambda))) \subset K^{(n)}\left(B_{r(\lambda)}(d_{n}(\lambda)) \right)
\end{equation}
where $$P(\lambda) = \frac{R^{2}a^{2}}{8M} =  \frac{\lambda(1-\epsilon)^{2}}{48}\ ; \ \ r(\lambda) = \frac{R^{2}a}{4M} =\frac{\lambda(1-\epsilon)}{24}$$
Thus , $ B_{P(\lambda)}(K^{(n)}(d_{n}(\lambda))) $ is in the domain of $(K^{(n)})^{\langle -1 \rangle}$.

Note that $ K^{(n)}(d_{n}(\lambda)) - d_{n}(\lambda) $ is diagonal.
Thus, in estimating the norm,  we need only estimate this in an entry-wise manner. 
To restrict to a single entry, we define $$k(w) = E((w^{-1} - X)^{-1}) : \mathcal{B}^{+} \mapsto \mathcal{B}^{+}$$
Estimating the norm, we have that
\begin{align}
\| k(i\lambda) - i\lambda \| & = \left\| E\left[ \left(\frac{1}{i\lambda} - X \right)^{-1} - i\lambda \right] \right\| \\
&=  \left\| \lambda E\left[ \int_{R} \frac{i\lambda}{1 - i\lambda t} -  \frac{i\lambda + \lambda^{2} t}{1 - i\lambda t} de_{X}(t) \right] \right\| \\
&= \label{iowa}  \left\| \lambda E\left[ \int_{R} \frac{-\lambda t}{1 - i\lambda t} de_{X}(t) \right] \right\| \\
&= o(\lambda)
\end{align}
uniformly for tight families of $X$ (since the integral converges to $0$ weakly, $E$ is weakly continuous, and the weak and norm topologies agree in finite dimensions).  Thus, for $\lambda_{n}$ small enough, 
\begin{equation}\label{othersetcontainment}
B_{p(\lambda)}(d_{n}(\lambda)) \subset  B_{P(\lambda)}(K^{(n)}(d_{n}(\lambda))) \subset \mathit{dom}\left( (K^{(n)})^{\langle -1 \rangle} \right) 
\end{equation}
 and the radius $p(\lambda)$ is uniform for tight families of $X$ and satisfies $P(\lambda) - p(\lambda) = o(\lambda)$ by \eqref{iowa}.

The claim follows for the $\mathcal{R}$-transform since $$(K^{(n)})^{\langle -1 \rangle}(w)^{-1} =  (G^{(n)})^{\langle -1 \rangle}(w)$$
so that these functions have the same domain.  This completes the proof of our theorem.
\end{proof}

We note that Theorem 1.4 in \cite{InvFun} does not apply to our setting since our bounds on the derivative are not uniform over $n$.
Moreover, we note here that the R-transform that we have generated is not a non-commutative function in the sense that its domain does not, in general, contain non-commutative open balls (that is, balls whose radii are uniform over $n$).  It may be possible for this to have a larger domain if we assume a second moment, but we cannot prove this at the time.

\begin{corollary}\label{Rdefines}
In the setting of the previous theorem, the function $\mathcal{R}_{X}$  restricted to $\sqcup_{n=1}^{\infty} B_{\epsilon_{n}}(d_{n}(\lambda_{n}))$   defines the operator algebra $\{X , \mathcal{B} \}''$ up to spatial isomorphism.
\end{corollary}
\begin{proof}
By assumption $(G_{X}^{(2n)})^{\langle -1 \rangle}$ has domain $B_{\epsilon_{n}}(d_{n}(\lambda_{n}))$ for some $\epsilon_{n} > 0$.  Since this is a noncommutative function, $$(G_{X}^{(2n)})^{\langle -1 \rangle}(d_{n}(\lambda_{n})) = \left[ \begin{array}  {ccccc}
    (G_{X}^{(1)})^{\langle -1 \rangle}(i) & 0 & 0 & \cdots & 0 \\
0 & (G_{X}^{(1)})^{\langle -1 \rangle}(-i) & 0 & \cdots & 0 \\
0 & 0  & (G_{X}^{(1)})^{\langle -1 \rangle}(i) & \cdots & 0 \\
\ & \vdots & \ & \vdots & \ \\
0 & 0  & 0 & \cdots & (G_{X}^{(1)})^{\langle -1 \rangle}(-i) \\
   \end{array} \right]$$
which is an element of $\Omega_{n}(\mathcal{B})$.  Thus, the image of $B_{\epsilon_{n}}(d_{n}(\lambda_{n}))$ contains a neighborhood of $$(G_{X}^{(2n)})^{\langle -1 \rangle}(d_{n}(\lambda_{n})).$$  Inverting this function, we recover the Cauchy transform $G_{X}^{(n)}$ on this set and, by analytic continuation and connectivity, we may recover it on $\Omega_{n}(\mathcal{B})$.  By Theorem \eqref{unboundiso}, we have our algebra isomorphism.
\end{proof}

We expect Theorem \eqref{Rexist} to be true for more general domains.  Indeed, if we replace $d_{n}(\lambda)$ with matrices with strictly imaginary operators on the diagonal, then the same should hold.  We chose this  domain since it has all of the properties that we required (connectivity, Corollary \eqref{Rdefines}, etc.) but with a greatly simplified proof.

\begin{theorem}\label{Rsum}
Assume that $(\mathcal{A}, E, \mathcal{B})$ is an operator valued, tracial  probability space where $\mathcal{B}$ is finite dimensional.
Let $X , Y \in \tilde{\mathcal{A}}^{sa}$ denote $\mathcal{B}$-free random variables.
Then $\mathcal{R}_{X}$ and $\mathcal{R}_{Y}$ exist on a fixed open subset of $\Omega$ and satisfy the equality
$$ \mathcal{R}_{X} + \mathcal{R}_{Y} = \mathcal{R}_{X+Y}. $$
\end{theorem}
\begin{proof}
Let 
\begin{equation}
p_{k} = e_{X}([-k,k]) \ ; \ \ q_{k} = e_{Y}([-k,k])  \ ; \ \ r_{k} = p_{k} \wedge q_{k} \ ; \ \ X_{k} = Xp_{k} \ ; \ \ Y_{k} = Yq_{k}.
\end{equation}
Observe that 
\begin{equation}\label{grant}
 r_{k}X = r_{k}p_{k}X = r_{k}X_{k} \ ; \ \ r_{k}Y = r_{k}q_{k}Y = r_{k}Y_{k}.
\end{equation}
We claim that 
\begin{equation}\label{Gclaim1}
G^{(2n)}_{X_{k}} \rightarrow G^{(2n)}_{X}
\end{equation}
\begin{equation}\label{Gclaim2}
G^{(2n)}_{X_{k} + Y_{k}} \rightarrow G^{(2n)}_{X + Y}
\end{equation}
on compact subsets of the $\Omega_{n}(\mathcal{B})$.

We proved \eqref{Gclaim1}  in line \eqref{P9} of the proof of Theorem \eqref{MainTheorem}.
Focusing on \eqref{Gclaim2}, 
we have that
\begin{equation}\label{Gclaim3}
G_{X_{k} + Y_{k}} - G_{X + Y} = [G_{X_{k} + Y_{k}} - G_{r_{k}X_{k} + r_{k}Y_{k}  }] + [ G_{r_{k}X + r_{k}Y } - G_{X + Y}]
\end{equation}
where this follows from \eqref{grant}.  Focusing on these terms individually, let $\{ p_{i} \}_{i=1}^{N(n)}$ generate the norm topology on $M_{n}(\mathcal{B})$ as vector states.

\begin{align}
\tau & \left( E\left[  (r_{k}(X_{k} + Y_{k}) - b)^{-1} - (X_{k} + Y_{k} - b)^{-1} \right] p_{i}\right) \\ \label{JHC1} &=\tau\left( (r_{k}(X_{k} + Y_{k}) - b)^{-1} (r_{k} - 1)(X_{k} + Y_{k}) (X_{k} + Y_{k} - b)^{-1}p_{i} \right) \\
&  \label{JHC2} \leq \tau(1 - r_{k})^{1/2}\tau\left( \left| (X_{k} + Y_{k}) (X_{k} + Y_{k} - b)^{-1}p_{i}   (r_{k}(X_{k} + Y_{k}) - b)^{-1}\right|^{2} \right)^{1/2}\\
&  \label{JHC3} \leq  \tau(1 - r_{k})^{1/2} \left[ 1 + \frac{\| \Im b \|}{\epsilon} \right] \|p_{i}\| \frac{1}{\epsilon}
\end{align}
where \eqref{JHC1} is bimodularity and contractivity of $E$, \eqref{JHC2} is Cauchy Schwarz and \eqref{JHC3} is Lemma \eqref{affresolvent} where we assume that $ |\Im{(b)}| > \epsilon1_{n}$  on the compact set from which the $b$ are drawn.
Since $$\tau(r_{k}) \geq \tau(p_{k}) + \tau(q_{k}) - 1 $$
we conclude that $\tau(r_{k})$ converges to $1$.  As we have a finite number of $b_{i}$, this proves that \eqref{JHC3} converges to $0$.  The proof where $X_{k} + Y_{k}$ are replaced by $X + Y$ is identical.  We conclude that \eqref{Gclaim3} converges to $0$, proving \eqref{Gclaim2}.

Now, assuming that $G_{X_{k}}$ converges to $G_{X}$ on compact subsets of the resolvent we show that
$G_{X_{k}}^{\langle -1 \rangle }$ converges to $G_{X}^{\langle -1 \rangle }$ uniformly the sets $B_{p(\lambda)}(d_{n}(\lambda))$ where the notation is lifted from Theorem \eqref{Rexist}.

In the context of Theorem \eqref{Rexist}, consider a point $d_{n}(\lambda)$ and $p(\lambda) > 0$
such that all of the relevant $\mathcal{R}$ transforms exist on $B_{p(\lambda)}(d_{n}(\lambda))$.
We may assume that $$(G_{X}^{(2n)})^{\langle -1 \rangle}(b) \leq M$$  on $B_{p(\lambda)/2}(d_{n}(\lambda))$ (due to analyticity).
Define $$ b_{k} = (G_{X_{k}}^{(2n)})^{\langle -1 \rangle}(b). $$
By \eqref{lipschitz}, have that
\begin{align*}
\| (G_{X}^{(2n)})^{\langle -1 \rangle}(b) - (G_{X_{k}}^{(2n)})^{\langle -1 \rangle}(b) \| & = \| (G_{X}^{(2n)})^{\langle -1 \rangle}\circ G_{X_{k}}^{(2n)}(b_{k})- (G_{X_{k}}^{(2n)})^{\langle -1 \rangle}\circ G_{X_{k}}^{(2n)}(b_{k}) \| \\
&= \|  (G_{X}^{(2n)})^{\langle -1 \rangle}\circ G_{X_{k}}^{(2n)}(b_{k}) - b_{k} \| \\
&= \|  (G_{X}^{(2n)})^{\langle -1 \rangle}\circ G_{X_{k}}^{(2n)}(b_{k}) - (G_{X}^{(2n)})^{\langle -1 \rangle}\circ G_{X}^{(2n)}(b_{k})  \| \\
& \leq \frac{2M \|G_{X_{k}}^{(2n)}(b_{k}) - G_{X}^{(2n)}(b_{k})\|}{p(\lambda)/2  - 2  \|G_{X_{k}}^{(2n)}(b_{k}) - G_{X}^{(2n)}(b_{k})\|}
\end{align*}
provided that \begin{equation}\label{domainstuff} G_{X}^{(2n)}(b_{k}) \in \mathit{dom}\left( (G_{X}^{(2n)})^{\langle -1 \rangle} \right).\end{equation}
But this holds since, after modifying \eqref{setcontainment} and \eqref{othersetcontainment} for the Cauchy transform instead of the function $K$, 
we may conclude that   $$(G^{(2n)}_{X_{k}})^{\langle -1 \rangle}(B_{p(\lambda)}(d_{n}(\lambda)) \subset B_{r(\lambda)}(d_{n}(\lambda))^{-1} .$$
In particular, if our $b$ is drawn from this set, then the $b_{k}$ are bounded.  As we have shown that the Cauchy transforms are convergent on bounded sets, we have that
$$ 0 = \lim_{k \uparrow \infty} \| G^{(2n)}_{X}(b_{k}) - G^{(2n)}_{X_{k}}(b_{k})\| =  \| G^{(2n)}_{X}(b_{k}) - b\|  $$
In particular, \eqref{domainstuff} holds so that
$$ \lim_{k \uparrow \infty } \| (G_{X}^{(2n)})^{\langle -1 \rangle}(b) - (G_{X_{k}}^{(2n)})^{\langle -1 \rangle}(b) \|  \rightarrow 0 $$
uniformly on $B_{p(\lambda)}(d_{n}(\lambda))$.  
Note that the same proof works for $X_{K} + Y_{k}$ since we only used that this is a tight family and the Cauchy transform converges to 
$G_{X + Y}$.

To complete the proof, consider the set $B_{p(\lambda)}(d_{n}(\lambda))$ where Theorem \eqref{Rexist} holds for the tight family of random variables
$$\{ X_{k}\}_{k \in \mathbb{N}} \cup \{ Y_{k}\}_{k \in \mathbb{N}}  \cup \{ X_{k} + Y_{k}\}_{k \in \mathbb{N}} \cup \{X , Y, X + Y \}.$$
For $b \in B_{p(\lambda)}(d_{n}(\lambda))$, we have
\begin{align*}
\| \mathcal{R}^{(2n)}_{X}(b) & +  \mathcal{R}^{(2n)}_{Y}(b) - \mathcal{R}^{(2n)}_{X + Y}(b)\| = \|(G_{X}^{(2n)})^{\langle -1 \rangle}(b) + (G_{Y}^{(2n)})^{\langle -1 \rangle}(b) - (G_{X + Y}^{(2n)})^{\langle -1 \rangle}(b) - b^{-1}  \| \\
&= \lim_{k \uparrow \infty}  \|(G_{X_{k}}^{(2n)})^{\langle -1 \rangle}(b) + (G_{Y_{k}}^{(2n)})^{\langle -1 \rangle}(b) - (G_{X_{k} + Y_{k}}^{(2n)})^{\langle -1 \rangle}(b) - b^{-1}  \| \\
&= \lim_{k \uparrow \infty}\| \mathcal{R}^{(2n)}_{X_{k}}(b)  +  \mathcal{R}^{(2n)}_{Y_{k}}(b) - \mathcal{R}^{(2n)}_{X_{k} + Y_{k}}(b)\| \equiv 0 
\end{align*}
where the last equality follows from the fact that our theorem is known for bounded operators.  This completes the proof.
\end{proof}

\section{Notes on the R-transform for Affiliated Operators for General B.}\label{RTransform2}




For infinite dimensional $\mathcal{B}$, it seems unlikely that a general theory of the $\mathcal{R}$ transform should be possible.
Indeed, consider an element $b' \in \tilde{\mathcal{B}}^{sa} \setminus \mathcal{B}^{sa}$.  If this set is non-empty then the analogue of the Dirac mass is a perfectly acceptable unbounded random variable.  In this case, letting $\mathcal{A} = B$ and $E = Id$, we trivially have an operator valued probability space.  The Cauchy transform of this Dirac mass is the map
$$ b\mapsto (b - b'\otimes1_{n})^{-1} $$ so that the $\mathcal{R}$ transform is simply the constant map
$$\mathcal{R}^{(n)}(b) = b'\otimes1_{n}. $$
Thus, as this simple example shows, the $\mathcal{R}$-transform may have range in the affiliated operators.
While this does not kill the theory, it does present challenges from an analytic standpoint.

Moreover, as we will show in the following example, if $\mathcal{B}$ is infinite dimensional, the $\mathcal{R}$-transform  need not exist, at least within the current conceptual framework.  The remainder of the section is dedicated to constructing an unbounded random variable with a nowhere locally invertible Cauchy transform.

\begin{lemma}
Let $z_{1} , z_{2} , \ldots , z_{n} \in \mathbb{C}^{+}$.  Then, there exists a Borel probability measure $\mu$ such that
$F_{\mu}'(z_{i})=0$ for $i= 1 , 2, \ldots, n$.  Moreover, the $F$ transform is not locally invertible at $z_{i}$.
\end{lemma}
\begin{proof}
Proceeding by induction, we begin with the $n=1$ case.
Consider the Bernoulli measure $$ \nu = \frac{ \delta_{r} + \delta_{-r}}{2}. $$
Observe that 
$$F_{\nu}(z) = \frac{z^{2} - r^{2}}{z} \ \ ; \ \ \ F_{\nu}'(z) = -\frac{z^{2} + r^{2}}{z^{2}}$$
so that $F_{\nu}'(\pm ir) = 0$.
Also note that this function maps the imaginary axis onto itself and the imaginary part has a global maximum at $ir$, showing that it is not locally invertible on any neighborhood of this point.
The family of $F$-transforms is closed under translation so, for any point $z = s + ir$,
the $F$-transform $F_{\nu}(z-s)$ has derivative vanishing at $z$.  

For the general case, select the points $z_{1} , z_{2} , \ldots , z_{n}$.
Let $\nu_{1}$ satisfy $F_{\nu_{1}}'(z_{1})=0$.
Moreover, for each $j = 2 , \ldots , n$, let  $\nu_{j}$ satisfy \begin{equation}\label{FZero}
F'_{\nu_{j}} \circ F_{\nu_{j-1}} \circ \cdots \circ F_{\nu_{2}} \circ F_{\nu_{1}}(z_{j}) = 0.
\end{equation}
Note that this step is possible since $F$-transforms map the upper-half plane into itself and the family of $F$-transforms is closed under composition.
Thus, for $$F_{\mu} := F_{\nu_{n}} \circ F_{\nu_{n-1}} \circ \cdots \circ F_{\nu_{1}} $$
our claim follows by the chain rule and \eqref{FZero}.  Non-invertibility on the neighborhoods of these points is immediate.
\end{proof}

\begin{example}
We  construct an example of an affiliated operator $X \in \tilde{ \mathcal{A}}^{sa}$ such that $G_{X}$ is not locally invertible at any point in $\mathcal{B}^{+}$. 
Indeed, let 
$$ \mathcal{A} := L^{\infty}((0,1), \mu) \otimes \ell^{\infty}(\mathbb{N}) \ ; \ \ \mathcal{B}:= 1 \otimes \ell^{\infty}(\mathbb{N}) \ ; \ \ \mathcal{H} := L^{2}((0,1), \mu) \otimes \ell^{2}(\mathbb{N}). $$
Let $\{ \xi_{n} \}_{n \in \mathbb{N}}$ denote  the support vectors on $\ell^{\infty}(\mathbb{N})$ and $\{ \psi_{m} \}_{m \in \mathbb{N}}$ denote the basis vectors for $\ell^{2}(\mathbb{N})$.
To define a trace $\mathcal{A}$, let $$\tau(f\otimes \xi_{n}) = \frac{\mu(f)}{2^{n}}$$  (this specific trace on $\mathcal{A}$ is not important) .
Define the conditional expectation in the obvious way
$$ E(f \otimes \xi_{n}) = \mu(f)1\otimes \xi_{n}. $$

Let $\{ z_{k} \}_{k \in \mathbb{N}} \subset \mathbb{C}^{+}$ denote a dense subset.  Let $\nu_{k}$ satisfy $F_{\nu_{k}}'(z_{p}) = 0$ for $p=1,2, \ldots , k$
and assume that this is realized by a random variable $X_{k} \in L^{\infty}((0,1), \mu)$.
Abusing notation, we define an affiliated operator $X = \sum_{k=1}^{\infty} X_{k}\otimes \xi_{k}$ with domain spanned by finite sums  $$\{ L^{2}((0,1), \mu) \otimes \psi_{k}\}_{k \leq K} $$
for all $K > 0$ in the obvious way.  This is a symmetric operator since it is clearly self adjoint on it's domain.
To show that this operator has self-adjoint extension, we need only show that
$$\overline{\mathcal{R}(X \pm iI)} = \mathcal{H}. $$
However, given a finite sum, $$\sum_{k=1}^{K} f_{k}\otimes \psi_{k} $$
in $\mathcal{H}$, observe that $X_{k} + i1$ is invertible for all $k$.  Defining
$$ g_{k} = (X_{k} + i1)^{-1}f_{k} \otimes \psi_{k}$$
we have that $$X \cdot \sum_{k=1}^{K} g_{k}\otimes \psi_{k}  =  \sum_{k=1}^{K} f_{k}\otimes \psi_{k} .$$
As this is a dense subspace of $\mathcal{H}$, our claim holds.

We now claim that $G_{X}$ is nowhere locally invertible on $\mathcal{B}$.
Indeed, pick a point $\{ w_{k} \otimes \xi_{k} \}_{k \in \mathbb{N}} \in \mathcal{B}$ where $\inf_{k}\Im(w_{k}) > 0$.
Consider an $\epsilon$ neighborhood of this point in the $\sup$ norm.
Let $w$ be a cluster point of the sequence $\{ w_{k} \}_{k \in \mathbb{N}}$.
Taking our dense subset, pick $m$ so that $|z_{m} - w| < \epsilon/2$.
Pick $p > m$ such that $|w_{p} - w| < \epsilon/2$.
By the triangle inequality, the point
$$ v = \{w_{1}\otimes \xi_{1} , \ldots , w_{p-1}\otimes \xi_{p-1} , z_{m}\otimes \xi_{p} , w_{m+ 1}\otimes \xi_{m + 1} , \ldots  \} $$
is contained in our $\epsilon$ neighborhood.

Now, $$ E((v - X)^{-1}) = \{ G_{\nu_{1}}(w_{1})\otimes \xi_{1} , \ldots , G_{\nu_{p-1}}(w_{p-1}) \otimes \xi_{p-1} ,  G_{\nu_{p}}(z_{m}) \otimes \xi_{p}  ,  G_{\nu_{p+1}}(w_{p+1}) \otimes \xi_{p+1}  , \ldots \}.$$
As $G_{\nu_{p}} $
is not locally invertible at $z_{m}$ since $m < p$, our claim follows.
\end{example}

\bibliographystyle{amsalpha}
\bibliography{Bfree}
\end{document}